\documentclass[12pt,leqno]{amsart}
\usepackage{tikz}
\usetikzlibrary{automata,positioning}
\usepackage[noBBpl]{mathpazo}
\usepackage{amssymb,amsthm,amsmath,amsfonts,latexsym}
\usepackage{microtype}
\usepackage{a4wide} 
\usepackage{graphicx} 
\graphicspath{{figures/}} 
\usepackage{booktabs} 
\usepackage[font=small,labelfont=bf]{caption} 
\usepackage{amsfonts, amsmath, amsthm, amssymb} 
\usepackage{wrapfig} 
\usepackage{multicol}
\usepackage{mathtools}
\newcommand\blfootnote[1]{%
  \begingroup
  \renewcommand\thefootnote{}\footnote{#1}%
  \addtocounter{footnote}{-1}%
  \endgroup
}
\usepackage{hyperref}

\newtheorem{theorem}{\sc Theorem}[section]

\newtheorem{thmx}{Theorem}

\newtheorem{question}[theorem]{Question}
\newtheorem{lemma}[theorem]{\sc Lemma}
\newtheorem{proposition}[theorem]{\sc Proposition}
\newtheorem{corollary}[theorem]{\sc Corollary}
\newtheorem{rem}[theorem]{\sc Remark}
\newtheorem{ex}[theorem]{\sc Example}

\newcommand{\T}{\mathcal{T}} 
\newcommand{\N}{\mathcal{N}}
\newcommand{\Am}{\mathcal{A}_{m}}  
\newcommand{\Amm}{\mathcal{A}_{m'}}
\newcommand{\Y}{\widetilde{Y}}
\newcommand{\Homeo}{\mathrm{Homeo}}  
\newcommand{\PSL}{\mathrm{PSL}_{2}}
\newcommand{\Symm}{\operatorname{Symm}}
\newcommand{\sgn}{\operatorname{sgn}}

\newcommand{\C}{\mathfrak{C}} 

\newcommand{\ZwrG}{\mathbb{Z}\wr_{X}G}

\title{On the Self-Similarity of Permutational Wreath Products and Their Embedding into Finitely Presented Simple Groups}

\author{Mailton Rego Almeida}
\address{Departamento de Matem\'atica, Universidade de Bras\'ilia,
Brasilia-DF, 70910-900 Brazil}
\email{(Almeida) mailtonalda16@gmail.com}

\author{Alex Carrazedo Dantas} 
\address{Departamento de Matem\'atica, Universidade de Bras\'ilia,
Brasilia-DF, 70910-900 Brazil}
\email{(Dantas) alexcdan@gmail.com}

\author{Altair Santos de Oliveira-Tosti}
\address{Colegiado de Matem\'atica, Universidade Estadual do Norte do Paran\'{a},
Cornélio Procópio-PR, 86.304-028, Brazil}
\email{(de Oliveira-Tosti) altair@uenp.edu.br \\ altairsot@tutanota.com}



\begin{document}

\blfootnote{The first author acknowledges support from the Brazilian scientific agency CAPES. The second author was supported by DPI, FAPDF and FEMAT. The third author thanks the Department of Mathematics at the University of Bras\'{i}lia for their hospitality.}

\begin{abstract}
    In this work, we study the self-similarity of permutational wreath products of the form \(A \wr_X G\), where \(A\) is a finitely generated abelian group and \(G\) is a self-similar group (the permutational wreath product \(A \wr_X G\) is also known as a lamplighter group). In the case where \(G\) is a non-torsion contracting group, we prove that, under certain conditions, a faithful self-similar representation of \(\mathbb{Z}^d \wr_X G\) yields an associated Scott--R\"over--Nekrashevych group \(V_m(\mathbb{Z}^d \wr_X G)\) that is finitely presented and virtually simple. Moreover, we show that this representation of \(\mathbb{Z}^d \wr_X G\) embeds into a finitely presented simple group.
\end{abstract}

\maketitle

\section{Introduction} Since Rostislav I. Grigorchuk \textit{et al.} gave a 
counterexample to a strong version of the Atiyah's conjectures about the range of 
$L^2$-Betti numbers of closed manifolds using a self-similar representation of the 
lamplighter group $C_2 \wr \mathbb{Z}$ \cite{GSLZ}, the study of self-similar 
representations of wreath products has attracted considerable attention. In 
particular, several works have appeared on self-similar representations of 
lamplighter groups $A \wr_{X} G$, where $A$ is a finitely generated abelian group 
and $G$ is a self-similar group. For instance, self-similar representations of 
regular wreath products $B \wr \mathbb{Z}^d$, where $B$ is a finite abelian group, 
were obtained in \cite{BSu,DS1,KSS,SS}.

Andrew M. Brunner and Said N. Sidki defined the tree-wreathing operation and 
proved that the group \(\mathbb{Z}\wr\mathbb{Z}\) has a finite-state representation 
as a subgroup of the automorphism group of the binary tree \cite{BruSid1}. Later, 
Laurent Bartholdi implemented the tree-wreath product in GAP and 
explicated this representation
\cite[Section: 7.1.10 TreeWreathProduct (FR group), p. 70]{Bar}.
However, this representation is not self-similar.
The second author and Sidki proved the group \(\mathbb{Z}\wr\mathbb{Z}\) has no
transitive self-similar representation \cite{DS} and, along with Túlio M. G. 
Santos, they exhibited non-transitive self-similar representations of the 
lamplighter group $\mathbb{Z}^l \wr \mathbb{Z}^d$ in \cite{DSS}. More recently, the 
papers \cite{BarSid, DOS, DSS} started a study of the self-similarity of permutational wreath products $A \wr_{X} G$, where $A$ is a finitely generated abelian group and 
$G$ is a self-similar group.

Before we present the statements of our results, let us consider a group $G$ with a 
self-similar representation induced by the virtual endomorphisms $f_i\colon H_i\to G$ for 
$1 \leq i \leq s$ (see Subsection \ref{Subsec:Virtual-Endomorphisms} for definition). A subgroup $K$ of $G$ \emph{satisfies condition $(\Omega)$} if
\begin{enumerate}
 \item [(i)] $H_{\omega}\leq K$,
 \item [(ii)]$(K\cap H_i)^{f_i}\leq K$, for every $i=1,\ldots,s$ and
 \item [(iii)] if $h^{f_i}\in K$, then $h\in K$,
\end{enumerate}
where $H_{\omega}=\langle K\leq\cap_{i = 1}^s H_{i} \mid K^{f_i} \leq K,~\forall~i=1,2,\dots,s\rangle$ is the parabolic subgroup. 

Our first result concerns the self-similarity of a lamplighter group of the
form \(A \wr_{K\setminus G} G\), where \(K\setminus G\) represents the right
cosets of \(K\) in \(G\).

\begin{thmx}\label{thmx:ZwrG-SelfSimilar}
 Let $G$ be a group with a self-similar representation induced by the virtual endomorphisms $f_i\colon H_i\to G$ for $1 \leq i \leq s$. Let $K\leq G$ be an infinite-index subgroup satisfying condition $(\Omega)$. Then,
 \begin{enumerate}
 \item the permutational wreath product $B\wr_{K\backslash G}G$ is self-similar, where $B$ is a finite abelian group.\smallskip
 \item the permutational wreath product $A\wr_{K\backslash G}G$ is self-similar, where $A$ is an infinite finitely generated abelian group and \(G\) is a non-torsion group.
 \end{enumerate}
\end{thmx}

Michael Kapovich affirmed that Zoran Šuni\'{c} showed that the projective special linear group $\PSL(\mathbb{Z})$ 
admits a transitive self-similar representation of degree \(3\) in the automorphism group of the ternary tree \cite{K}. We recall that strongly self-similar
means that the parabolic subgroup is trivial. Motivated by this result, 
we prove Theorem \ref{thmx:PSL-FinitelyPresentedTransitiveSelfSimilar}, stated 
below. 

\begin{thmx}\label{thmx:PSL-FinitelyPresentedTransitiveSelfSimilar}
Given a prime number \(p\), it holds that
\begin{enumerate}
    \item $\PSL(\mathbb{Z}[1/p])$ is a transitive self-similar group if \(p\geq3\);
    \item $\PSL(\mathbb{Z}[1/2])$ is a transitive strongly self-similar group; 
    \item Since \(\PSL(\mathbb{Z})\) satisfies the conditions \((\Omega)\), \(A\wr_{X}\PSL(\mathbb{Z}[1/2])\) is a self-similar group, where \(A\) is a finitely generated abelian group and 
    \(X=\PSL(\mathbb{Z})\setminus \PSL(\mathbb{Z}[1/2])\).
\end{enumerate}
\end{thmx}

It is worth pointing out that the group \(\mathbb{Z}\wr_{X}\PSL(\mathbb{Z}[1/2])\)
is mentioned by Yves Cornulier as an example of non-finitely presented residually 
finite wreath product on which \(\PSL(\mathbb{Z}[1/2])\) acts primitively and with finitely generated stabilizers over 
\(X=\PSL(\mathbb{Z})\setminus \PSL(\mathbb{Z}[1/2])\) \cite[p. 106]{Cornulier2006}.

Recently, Dessislava H. Kochloukova and Sidki
studied the Dantas--Sidki example, that is, the transitive self-similar representation of the lamplighter group
\(C_{p}\wr\mathbb{Z}^{d}\) (see \cite{DS1}), having extended this construction to a finitely
presented metabelian self-similar group that contains a copy of \(C_{p} \wr \mathbb{Z}^{d}\) as a subgroup
\cite{KS}. Considering that Matthew C. B. Zaremsky proved that every finitely presented self-similar group embeds in a finitely presented simple group \cite{Zaremsky2025b}, it follows that the Kochloukova--Sidki construction
embeds into a finitely presented simple group and satisfies the 
Boone--Higman conjecture (which asserts that every finitely generated group has solvable word problem if and only if it embeds into a finitely presented simple group \cite{Boone1973,BooneHigman1974}). On the other hand, James Belk and Francesco Matucci
showed that every contracting self-similar group embeds into a
finitely presented simple group (hence satisfies Boone--Higman conjecture) \cite{BelkMatucci2025}. In this sense, our next result
provides examples of self-similar groups which satisfy the result of Belk and Matucci. First, let us recall that the self-similar 
representation of \(G\) depends on virtual endomorphisms \(f_{i}\) 
and the transversals $T_i=\{t_{i1}, t_{i2},\dots,t_{im_i}\}$ of \(H_{i}\) in 
\(G\), where \(1\leq i \leq s\). In Subsection \ref{Subsec4.1}, we introduce examples of this result.

\begin{thmx}\label{thm:BGcontratil}
Let \(G\) be a contracting self-similar group that satisfies the 
hypotheses of Theorem~\ref{thmx:ZwrG-SelfSimilar}.
Suppose that \(G\) is self-similar with respect to the endomorphisms \(f_i\) and the transversals, $T_i = \{t_{i1}, t_{i2},...,t_{im_i}\}$, where $t_{ij}$ are chosen in $K \sqcup \left(G\backslash H_{i}K \right)$, and that $K \leq N_G(H_i)$ for every $1\leq i\leq s$, 
then the wreath product 
$B\wr_{K\backslash G}G$ is a contracting self-similar group, for every finite abelian group \(B\).
\end{thmx}

The mentioned embeddings provided by Zaremsky in \cite{Zaremsky2025b} 
and by Belk and Matucci in \cite{BelkMatucci2025} were constructed 
from the \emph{Scott--R\"over--Nekrashevych group}\footnote{In 
literature, they are usually called R\"{o}ver--Nekrashevych groups.}
(SRN-group, for short) associated to a self-similar automorphism group \(G\) of the one-rooted \(m\)-ary tree, with \(m \ge 2\). This kind of group, denoted by \(V_m(G)\), 
is a group of homeomorphisms of the Cantor set generated by Higman--Thompson group \(V_m\) (see \cite{Higman1974}) and the self-similar 
group \(G\), as we briefly recall in Subsection~\ref{Subsec:SRN-groups}. These groups have attracted attention of researchers, who 
have used them to prove the existence of simple groups with certain 
finiteness properties 
\cite{LlosaIsenrichScheslerWu2025,SkipperWitzelZaremsky2019} and to
construct the first examples of groups with at least exponential
Dehn functions \cite{Zaremsky2026}. Moreover, these groups are 
important examples of rational similarity groups (RSGs), introduced  by Belk, Collin Bleak, Matucci, and Zaremsky in \cite{BelkBleakMatucciZaremsky2026}, on which they proved that the Boone--Higman conjecture holds for all hyperbolic groups.

The case of embedding of \(\mathbb{Z}^{d}\wr_{X} G\) into a finitely
presented simple group may not happen as \(B\wr_{X}G\). For instance,
Kochloukova and Melissa S. Luiz examined several families of finitely 
presented metabelian groups which contain \(\mathbb{Z} \wr \mathbb{Z}\) as a subgroup but are not self-similar \cite{KSL}. 
In this direction, we provide a family of groups for which the
faithful self-similar representation yields an associated SRN-group
that is finitely presented and virtually simple. Specifically, in Theorem \ref{thmx:SRN-FinitePresentation}, we prove that the permutational wreath
product \(\mathbb{Z}^{d} \wr_X G\), where \(G\) is a non-torsion 
contracting self-similar group, embeds into a finitely
presented simple group. To the best of our 
knowledge, this is the first result establishing this for a
specific class of permutational wreath products.

\begin{thmx}\label{thmx:SRN-FinitePresentation}
 Let \(G\) be a non-torsion contracting self-similar group satisfying the
 hypotheses of Theorem~\ref{thmx:ZwrG-SelfSimilar}. Then,
 \begin{enumerate}
 \item the SRN-group \(V_{d(m+1)}(\mathbb{Z}^{d}\wr_{X}G)\) is finitely presented and has finite abelianization.\smallskip
 \item the permutational wreath product \(\mathbb{Z}^{d}\wr_X G\) embeds into a finitely presented simple group.
 \end{enumerate}
\end{thmx}

 This theorem has immediate consequences for the 
 Boone--Higman conjecture. In the case where the 
 permutational wreath product \(\mathbb{Z}^{d} \wr_{X} G\) is a finitely generated group, it follows from the Boone--Higman theorem, which characterizes finitely generated groups with a solvable word
problem as precisely those that embed into a simple subgroup of a 
finitely presented group \cite[Theorem 1]{BooneHigman1974}, that it
has solvable word problem, and hence satisfies the Boone--Higman conjecture. Moreover, if \(G\) is a non-torsion contracting self-similar group satisfying the assumptions of Theorems \ref{thmx:ZwrG-SelfSimilar} and \ref{thm:BGcontratil}, it follows from Proposition 5.2 of \cite{DSS}, which establishes the corresponding result for the normal wreath product, that the permutational wreath product \(A\wr_{X}G\) embeds into a finitely presented simple group, and hence satisfies the Boone--Higman conjecture, for every infinite finitely generated abelian group \(A\).

We point out that the group \(\mathbb{Z}\wr_{X}\PSL(\mathbb{Z}[1/2])\),
from Theorem \ref{thmx:PSL-FinitelyPresentedTransitiveSelfSimilar},
does not satisfy our hypothesis from Theorem \ref{thmx:SRN-FinitePresentation}, nor from Zaremsky's work and Belk and Matucci results. That said, we ask

\begin{question}
    Does there exist an embed of \(\mathbb{Z}\wr_{X}\PSL(\mathbb{Z}[1/2])\) into a finitely presented self-similar group?
\end{question}

\begin{question}
    What other contable self-similar groups \(G\) embed into a SRN-group \(V_{m}(G)\) with finitely presented derived subgroup?
\end{question}

\noindent and 

\begin{question}
    What other permutational wreath products \(A\wr_{X}G\) embed into a
    self-similar group \(\mathcal{G}\) such that the commutator
    group \(\bigl[V_{m}(\mathcal{G}),V_{m}(\mathcal{G})\bigr]\) is finitely presented?
\end{question}

The paper is organized as follows. In Section \ref{Section:Preliminaries}, we recall
some preliminaries on self-similar groups, virtual endomorphisms, permutational wreath products and SRN-groups. In Section \ref{Section:ProvaThmA}, we
prove Theorem \ref{thmx:ZwrG-SelfSimilar}. In Section \ref{Section:SelfSimilar-PSL},
we discuss the self-similarity of some projective special linear groups,
by proving Theorems \ref{thmx:PSL-FinitelyPresentedTransitiveSelfSimilar} and \ref{thm:BGcontratil},
and applications of Theorem \ref{thmx:ZwrG-SelfSimilar}. Finally, in Section \ref{Section:ProvaUltimosTeoremas}, we
prove Theorem \ref{thmx:SRN-FinitePresentation} and provide some examples.

\section{Preliminaries}\label{Section:Preliminaries}

\subsection{Self-similar groups}\label{Subsec:Self-similarity}
For an integer \(m\geq2\), let \(Y=\{1,\dots,m\}\) be a finite set, which we call
\emph{alphabet}. The set of all finite words over \(Y\), including the empty-word
\(\emptyset\), is denoted by \(Y^*\) and it has the structure of a rooted \(m\)-ary tree,
denoted by \(\T_m\), on which the empty word is the root and the edges are the pairs
\((u,uy)\), where \(u\in Y^*\) and \(y\in Y\). In this case, \(u\) is a \emph{prefix} of
\(uy\), and \(uy\) is a \emph{descendant} of \(u\). The length of a word \(v=y_{1}\dots y_{k}\) is the number of its letters and we denote it by \(|v|\). The level \(k\) is
the set of all words of length \(k\), denoted by \(Y^{k}\). Two vertices are
\emph{incomparable} if neither is a descendant of the other. The boundary of \(\T_m\) is
the Cantor space \(\C_m=Y^\omega\). For each \(v\in Y^*\), the set \(C_v\coloneqq\{v\eta\mid\eta\in\C_m\}\) is the \emph{cone} at \(v\) on
\(\C_{m}\).

The automorphism group of \(\T_m\), denoted by \(\Am\), is isomorphic to the
permutational restricted wreath product recursively defined as
\(\Am=\Am \wr_{Y}\Symm(Y)\), where \(\Symm(Y)\) represents the
symmetric group of degree \(m\). An automorphism \(\alpha\in\Am\) has the form
\(\alpha = (\alpha_1, \dots, \alpha_m) \sigma(\alpha)\), where \(\alpha_i\in\Am\) are \emph{states} of \(\alpha\) and \(\sigma\colon\Am\to\Symm(Y)\) is the permutational representation of
\(\Am\), the first level of the tree \(\T_m\). This form is called the \emph{decomposition
of \(g\) at the first level}.
For \(k \ge 1\), the action of \(\alpha\) on a word \(y_1 y_2 \cdots y_k \in Y^k\) is
given as follows
\[\alpha\colon y_1 y_2 \cdots y_k \mapsto (y_1)^{\sigma(\alpha)} (y_2 \cdots y_k)^{\alpha_{y_1}}.\]
Given \(\alpha\in \Am\), the set of automorphisms
\[Q(\alpha)=\{\alpha\}\cup Q(\alpha_{1})\cup \dots \cup Q(\alpha_{m})\]
is called the \emph{set of states of \(\alpha\)}. If \(Q(\alpha)\) is finite, we say that \(\alpha\) is \emph{finite-state}. If every
element from \(G\leq\Am\) is finite-state, we say that
\(G\) is finite-state.
A subgroup \(G\leq\Am\) is \emph{self-similar (state-closed)} if the set
\(Q(\alpha)\subseteq G\) for each element \(\alpha\in G\).

\subsection{Virtual endomorphisms and self-similarity of groups}\label{Subsec:Virtual-Endomorphisms} A \emph{virtual endomorphism} of a group \(G\) is a homomorphism
\(f\colon H\to G\) from a finite-index subgroup \(H\leq G\). Given 
\(H_i\) in \(G\), let
us consider the virtual endomorphisms
\(f_i\colon H_i\to G\), where \(1\leq i\leq s\). Defining \(m_i\coloneqq [G:H_i]\), we
fix a right transversal \(T_i=\{t_{i1},\dots,t_{im_i}\}\) for each \(H_i\leq G\).
Let \(\theta_i\colon G \times T_i \to H_i\), defined by
\[\theta_i(g,t_{ij})=t_{ij}g(t_{ik})^{-1},\]
where \(H_it_{ik}=H_it_{ij}g\), be the \(i\)-Schreier function and let \(\sigma\) be the induced permutation representation of \(G\) on \(T=\bigcup_{i=1}^{s}T_{i}\), which means
that \(j^\sigma=k\) if and only if \(H_it_{ik}=H_it_{ij}g\) for some \(i=1,\cdots,s\). The maps $f_i$ induce a homomorphism
\[\varphi\colon G \to \Am, g^{\varphi}=(\theta_i(g,t)^{f_i \varphi}\mid t\in T_i, 1\leq i \leq s)~g^\sigma.\]
Essentially, this is the statement of the Kaloujnine--Krasner Theorem \cite{KK}.
We say that a normal subgroup \(K\trianglelefteq G\) is \emph{$\mathbf{f}$-invariant}
if \(K\leq \bigcap_{i=1}^s H_i\) and
\[K^{f_i}\leq K,~\forall~i=1,\dots,s,\]
where $\mathbf{f}=(f_1,\dots,f_s)$. The subgroup generated by all $\mathbf{f}$-invariant normal subgroups is the \textit{$\mathbf{f}$-core}. If the $\mathbf{f}$-core is trivial, then the action \(\varphi\) is \textit{faithful}. A group \(G\) is transitive self-similar if it admits a faithful state-closed representation on a regular rooted tree such that the induced action of \(G\) on the first level is transitive. Similarly, \(G\) is intransitive self-similar if it admits a faithful state-closed representation on a regular rooted tree whose action on the first level is not transitive. Note, however, that the intransitive self-similarity of \(G\) does not preclude the existence of a transitive self-similar 
representation. The following result gives us a characterizations of self-similar groups via \(\mathbf{f}\)-core.
\begin{proposition}{\cite[Proposition A]{DSS}} A group $G$ is intransitive self-similar if and only if there exist virtual endomorphisms $f_i\colon H_i\to G$ for $1\leq i\leq s$ such that $\mathbf{f}$-core is trivial.
\end{proposition}

If $s=1$, under the conditions above, we say \textit{$f$-invariant} instead of $\mathbf{f}$-invariant, and \textit{$f$-core} instead of $\mathbf{f}$-core. In this case,
a characterization of self-similar groups is given by the next result.

\begin{proposition}{\cite[Theorem 3.1]{NS}}
A group $G$ is transitive self-similar if and only if there is a virtual endomorphism $f\colon H \to G$ such that the $f$-core is trivial.
\end{proposition}

A group $G\leq \Am$ is \emph{contracting self-similar} if there exists a finite set $S_G \subset G$ such that for every $g \in G$ there exists $N(g)\in \mathbb{N}$ such that $g|_v \in S_G$ for all word $v \in Y^{*}$ with $|v| \geq N(g)$. Equivalently, \(G\) is contracting self-similar with respect to the
virtual endomorphisms $f_i$ and the transversal $T_i$, where $1 \leq i \leq s$, if there
exists a finite subset $S_G \subseteq G$ such that, for every $g \in G$, there exists an integer $N(g)$ satisfying
\[S_n(g)\subseteq S_G,~\forall~n\ge N(g).\]
where $S_0(g)=\{g\}$, and for $n \geq 0$, $S_{n+1}(g) = \left\{ \theta_i(x,t)^{f_i}\mid x\in S_n(g),~t \in T_i,~1\le i\le s \right\}$. The minimal set $S_G$ satisfying this condition is called the \emph{nucleus} and it is denoted by $\mathcal{N}(G)$. For more information on contracting self-similar groups, we refer the reader to \cite[Section 2.11]{Nekra-SelfSimilar}.

Now, if $L$ is a subset of $G$, let us consider the subset $L^{f_{i}^{-1}}=\{h \in H_{i} \mid h^{f_i} \in L\}$ of $H_{i}$. We say that
the subgroup $H_{\omega} = \bigcap_{k=0}^{\infty} W_{k}$, where
\[W_{0} = G, W_{1} = \bigcap_{i=1}^{s} W_{0}^{f_{i}^{-1}}, W_{2} = \bigcap_{i=1}^{s} W_{1}^{f_{i}^{-1}}, \dots, W_{k} = \bigcap_{i=1}^{s} W_{k-1}^{f_{i}^{-1}},\]
is the \emph{parabolic subgroup} of $G$. It is immediate that
\[H_{\omega} = \langle K \leq \bigcap_{i=1}^{s} H_{i} \mid K^{f_{i}} \leq K,~\forall~i = 1, \dots, s \rangle.\]
Note that the $\mathbf{f}\text{-core}$ is a subgroup of $H_{\omega}$. We say that the group $G$ is \emph{strongly self-similar} if $H_{\omega}$ is trivial.

\subsection{Permutational restricted wreath products}\label{FA.de.ZePSL}
Let $W$ and $G$ be groups and let $X$ be a $G$-set. Let us consider the restricted direct
product $W^{(X)} = \bigoplus_{x \in X}W$, the group of functions $f\colon X \rightarrow W$ with finite support, equipped with the point-wise multiplication. There is a natural
action of the group $G$ on $W^{(X)}$ which permutes the indices as follows: for $g \in G$ and $f \in W^{(X)}$ define $(x)(f\cdot g)=((x)g^{-1})f$, for all $x \in X.$
The permutational wreath product of $W$ by $G$, denoted by
\[W \wr_XG,\]
is defined as the semidirect product $W^{(X)}\rtimes G$, with respect to this action. Let $y$ be an element of the permutational wreath product $W \wr_X G $. Then $y$ can be written in the form $b \cdot g$, where $b \in W^{(X)}$ and $g \in G$. If we fix some indexing $\{x_1,..., x_d\}$ of the set $X$, then $b$ can be written as $w_1^{x_1}...w_l^{x_l}$ and
\[
y = (w_1^{x_1} \dots w_l^{x_l}) g,
\]
where each $w_i \in W $ and each $x_i \in X$ indicates the location (or index) of $w_i$. In the particular case where $X = K\setminus G $, the set of right cosets of a subgroup $K \leq G $, we write
\[
x = (w_1^{K g_1} \dots w_l^{K g_l}) g.
\]
If, furthermore, $W = \mathbb{Z}$ and $a$ is a generator of $\mathbb{Z}$, then we write
\[
x = (a^{n_1 K g_1} \dots a^{n_l K g_l}) g = ( a^{\, n_1 K g_1 + \dots + n_l K g_l} ) g.
\]

\subsection{Scott--R\"over--Nekrashevych groups}\label{Subsec:SRN-groups}
Let \(\C_{m}=Y^{\omega}\)
denote the Cantor space. It is the boundary at infinity of \(\T_m\). For each
\(v\in Y^{*}\), the set \(C_{v}\) of all sequences which have $v$ as prefix is called \emph{cone on \(v\)} in \(\C_{m}\), that is,
\[C_{v}\coloneqq\{v\eta\mid \eta\in\C_{m}\}.\]
Every open set of \(\C_{m}\) is a disjoint union of cones and each
clopen set of \(\C_{m}\) is a finite disjoint union of cones \cite{BelkMatucci2025}. A homeomorphism
\(f\in \Homeo(\C_{m})\) is an element of \(V_{m}(G)\) if and only if there exist finite partitions
\(C_{v_1},\ldots,C_{v_n}\) and \(C_{u_1},\ldots,C_{u_n}\) of \(\C_{m}\) into disjoint cones and elements
\(\alpha_1,\ldots,\alpha_n\in G\) such that
\[(v_{i}w)f = u_i\,(w)^{\alpha_i},\]
for each $1\leq i\leq n$ and all $w\in Y^\omega$. Such element can be represented as the following table
\[f=\begin{pmatrix}
 v_{1}\quad v_{2}\quad\ldots\quad v_{n} \\ \alpha_{1}\quad \alpha_{2}\quad\ldots\quad\alpha_{n} \\ u_{1}\quad u_{2}\quad \ldots\quad u_{n}
\end{pmatrix},\]
where \(\alpha_{1},\ldots,\alpha_{n}\in G\). For an explicit definition, we refer the reader to either \cite[Definition~2.2]{LlosaIsenrichScheslerWu2025} or
\cite[Definition~2.2]{Zaremsky2025b}.

The first construction of these kind of group due to Elizabeth A. Scott in \cite{Scott1984}, where she presented
a procedure which allows constructing finitely presented infinite simple groups and extends the work of
Higman \cite{Higman1974}. Scott incidentally proved that if
\(G\leq\Am\) is a finitely presented self-similar group, then \(V_{m}(G)\) is finitely presented \cite[Theorem 2]{Scott1984}.
Later, Claas E. R\"over extended her method to develop finitely presented simple groups on which Grigorchuk
groups are embedded (see \cite{Rover1999}). Volodymyr V. Nekrashevych defined an analogue
of the Higman-Thompson groups for the Cuntz-Pimsner algebra of a self-similar action, extending
R\"over's work (see \cite{Nek2004}). Later, Nekrashevych proved that
if \(G\leq \Am\) is a self-similar group,
the commutator subgroup \([V_{m}(G),V_{m}(G)]\) is simple \cite[Theorem 4.7]{Nek2017} and, if \(G\) is contracting self-similar,
\(V_{m}(G)\) is finitely presented \cite[Theorem 5.9]{Nek2017}. In this direction of finiteness properties, Rachel Skipper, Stefan Witzel and Zaremsky proved that the group \(V_m(G)\) is of type \(\mathrm{F}_n\) whenever \(G\) is of type \(\mathrm{F}_n\) \cite[Theorem 4.15]{SkipperWitzelZaremsky2019} and Claudio Llosa Isenrich, Eduard Schesler and Xiaolei Wu demonstrated that
\(V_{m}(G)\) is of type \(\mathrm{FP}_{n}(R)\), for some unital commutative ring \(R\), if \(G\) is of type \(\mathrm{FP}_n(R)\) \cite[Proposition 3.1]{LlosaIsenrichScheslerWu2025}.

\section{Proof of Theorem \ref{thmx:ZwrG-SelfSimilar}}\label{Section:ProvaThmA}
In this section, we establish results concerning the self-similarity of permutational wreath products.
\subsection{Proof of Theorem A}

First, we show that $\mathbb{Z}\wr_{K\backslash G}G$ is self-similar. To this end, we begin by establishing the following auxiliary lemma.

\begin{lemma}\label{função_lambda}
Let $G$ be a group with a self-similar representation induced by the virtual endomorphisms $f_i\colon H_i\to G$ for $1 \leq i \leq s$. Let $K\leq G$ be an infinite-index subgroup satisfying condition $(\Omega)$.
For each $1\leq i \leq s$, define $L_i=K\cap H_i$ and consider the map
\[
\lambda_i\colon L_i\backslash H_i\longrightarrow K\backslash G,
\qquad
L_i h\longmapsto K h^{f_i}.
\]
Then $\lambda_i$ is well defined and injective.
\end{lemma}
\begin{proof}
We first show that \(\lambda_i\) is well defined. Indeed, suppose that $L_i h_1=L_i h_2$.
Then $h_1h_2^{-1}\in L_i=K\cap H_i$. By hypothesis (ii), $(h_1h_2^{-1})^{f_i}\in K$, that is, $Kh_1^{f_i}=Kh_2^{f_i}$. Hence, $(L_i h_1)^{\lambda_i}=(L_i h_2)^{\lambda_i}$, as required.

\medskip We now prove that \(\lambda_i\) is injective. Suppose that $(L_i h_1)^{\lambda_i} = (L_i h_2)^{\lambda_i}$. Then $Kh_1^{f_i}=Kh_2^{f_i}$, which implies that $(h_1h_2^{-1})^{f_i}\in K$. By hypothesis (iii), it follows that $h_1h_2^{-1}\in K$.
Since \(h_1, h_2\in H_i\), we obtain $h_1h_2^{-1}\in K\cap H_i=L_i$, and therefore $L_i h_1=L_ih_2$. Consequently, \(\lambda_i\) is injective for every \(i=1,\ldots,s\).

\end{proof}

Consider $\mathcal{G} = \mathbb{Z} \wr_{K \setminus G} G = \langle a \rangle \wr_{K \setminus G} G = \langle a^{K} \rangle ^{G} \rtimes G$. For each $i \in \{1, ..., s\}$, define the subgroup
$$\mathcal{H}_i =\langle a^{K}\rangle^{G} \rtimes H_i,$$
and let
$$\mathcal{H} = \mathcal{G}.$$
Note that the index $[\mathcal{G}:\mathcal{H}_i] = m_i$. For each $i \in \{1, 2, \dots, s\}$ define the homomorphism $\rho_{i}: \mathcal{H}_i \rightarrow \mathcal{G}$ that extends the map

\begin{align*}
 \rho_{i}\colon a^{Kg}&\mapsto
\begin{cases}
a^{Kh^{f_i}},~\text{if}~Kg=Kh~\text{for some}~h\in H_i \\
e, \textrm{ otherwise}
 \end{cases}\\
 h &\mapsto h^{f_i},~h\in H_i.
\end{align*}
We also define the homomorphism $\rho: \mathcal{H} \rightarrow \mathcal{G}$ as the extension of
\begin{align*}
 \rho\colon a^{Kg}&\mapsto x,\\
 g&\mapsto e,
\end{align*}
where \(g,x\in G\) and \(o(x)=\infty\).
The proof that the maps $\rho_i$, for $1 \leq i \leq s$, and $\rho$ are well defined and are homomorphisms is deferred to the following lemma.
\begin{lemma}
 The maps $\rho_i$, for $1 \leq i \leq s$, and $\rho$ are well defined and are homomorphisms.
\end{lemma}
\noindent\textit{Proof of the Lemma.} Clearly, $\rho$ is well defined and is a homomorphism. For the maps $\rho_i$, we begin by proving that they are well defined. To this end, it suffices to verify that
$$a^{Kg} = a^{Kg'}\quad\Longrightarrow\quad \bigl(a^{Kg}\bigr)^{\rho_i} =\bigl(a^{Kg'}\bigr)^{\rho_i} $$
Since $a^{Kg} = a^{Kg'}$, it follows that $Kg = Kg'$.

First, suppose that there exists $h\in H_i$ such that $Kh=Kg=Kg'.$ By the definition of $\rho_i$, we obtain
\[
\left(a^{Kg}\right)^{\rho_i}
=
a^{Kh^{f_i}}
=
\left(a^{Kg'}\right)^{\rho_i}.
\]
Now suppose that there is no element $h\in H_i$ satisfying $Kh=Kg=Kg'$.
In this case,
\[
\left(a^{Kg}\right)^{\rho_i}
=
e
=
\left(a^{Kg'}\right)^{\rho_i},
\]
showing once again that the images coincide.

It remains to verify that the definition is independent of the choice of representative in $H_i$. Let $h,h'\in H_i$ be such that $Kh=Kh'$. Then $hh'^{-1}\in K$, and consequently,
\[
hh'^{-1}\in K\cap H_i.
\]
By condition $(\Omega)$, it follows that $(hh'^{-1})^{f_i}\in K.$
Hence, $Kh^{f_i}=Kh'^{f_i}$,
and therefore,
\[
a^{Kh^{f_i}}
=
a^{Kh'^{f_i}}.
\]
Thus, the maps $\rho_i$ are well defined.

It remains to show that the maps $\rho_i$ are homomorphisms. First, observe that the restriction of $\rho_i$ to the subgroup $\langle a^K \rangle^G$ is given by
\[
\rho_i\big|_{\langle a^K \rangle^G}\colon\langle a^K \rangle^G\longrightarrow \langle a^K \rangle^G
\]
and that the restriction of $\rho_i$ to the subgroup $H_i$ coincides with
\[
f_i=\rho_i\big|_{H_i}\colon H_i\longrightarrow H_i.
\]
Since both restrictions are homomorphisms, it suffices to verify that, for every
$a^{Kg} \in \langle a^K \rangle^G$ and every $h\in H_i$, the identity
\[
\left((a^{Kg})^h\right)^{\rho_i}
=
\left((a^{Kg})^{\rho_i}\right)^{h^\rho_i}.
\]
holds.

On the one hand,
\[
\left((a^{Kg})^h\right)^{\rho_i}
= \left(a^{Kgh}\right)^{\rho_i} =
\begin{cases}
a^{K(h'h)^{f_i}}, & \text{if } Kg\in KH_i,\\[1ex]
e, & \text{if } Kg\notin KH_i.
\end{cases}
\]
where $Kg = Kh'$ for some $h' \in H_i$. On the other hand,
\[
\left((a^{Kg})^{\rho_i}\right)^{h^{\rho_i}}
=
\begin{cases}
\left(a^{K(h')^{f_i}}\right)^{h^{f_i}}= a^{K(h'h)^{f_i}}, & \text{if } Kg\in KH_i,\\[1ex]
e^{h^{f_i}} = e, & \text{if } Kg\notin KH_i.
\end{cases}
\]

Comparing the two expressions, we obtain the desired equality. Hence, $\rho_i$ preserves the action of $H_i$ on $\langle a^K \rangle^G$, and therefore is a homomorphism. This completes the proof.

\hfill$\square$

Returning to the proof of the theorem, we claim that the self-similar representation of $\mathcal{G}$ induced by the virtual endomorphisms $\rho_i\colon\mathcal{H}_i\longrightarrow\mathcal{G}$ and $\rho\colon\mathcal{H}\longrightarrow\mathcal{G}$ is faithful. Let $N\leq
\mathcal{H}
\cap
\bigcap_{i=1}^{s}\mathcal{H}_i$ be a normal subgroup of \(\mathcal{G}\) such that
\[
N^{\rho}\leq N
\qquad\text{and}\qquad
N^{\rho_i}\leq N,
\quad i=1,\ldots,s.
\]
Since $G$ is state-closed with respect to the maps $f_i$, it follows that $N\leq\langle a^{K}\rangle^{G}$. Assume, by contradiction, that there exists $y\in N\setminus\{1\}$ with minimal support $|y|=r>1$. Write
\[
y=a^{n_1K h_1+n_2K h_2+\dots+n_rK h_r}.
\]
By the minimality of the support and the normality of $N$, we may assume that $h_1 = 1, h_2,...,h_r \in \bigcap_{i=1}^sH_i$. Since $ \lambda_i$ is injective by Lemma \ref{função_lambda} and \(r\) was chosen to be minimal, it follows that $y^{\rho_{i_1}\rho_{i_2}\dots\rho_{i_k}}\neq1$
and
\[
\left|
y^{\rho_{i_1}\rho_{i_2}\dots\rho_{i_k}}
\right|
=
\left|
a^{\,n_1K+n_2Kh_2^{f_{i_1}\dots f_{i_k}}
+\cdots+
n_rKh_r^{f_{i_1}\dots f_{i_k}}}
\right|
=r,
\]
for every $k$. We conclude that $h_2,\ldots,h_r\in KH_{\omega}$. Since, by hypothesis, $H_{\omega}\leq K$,
it follows that $h_2,\ldots,h_r\in K$. Hence,
\[ y =a^{(n_1+n_2+\cdots+n_r)K}.\]
Therefore, $r=1$, contradicting the assumption that $r > 1$. We conclude that $N=\{e\}$, and the induced representation is faithful. Consequently,
\[
\mathbb{Z}\wr_{K\backslash G}G
\]
is self-similar.

\(\hfill \square\)
\subsubsection{Proof of the item (1)} Let $f_i\colon H_i\to G$, for $1 \leq i \leq s$, be the virtual endomorphisms inducing a self-similar representation of $G$. Consider the permutational wreath product
\[
\mathcal{G}=B\wr_{K\backslash G}G
=
\bigoplus_{K\backslash G}B\rtimes G
= \mathcal{B} \rtimes G.\]
For each $i \in \{1,\dots, s\}$, define the subgroup of $\mathcal{G}$ by
$$\mathcal{H}_i=[B^K,G]H_i = \langle [a^{K},\,g ] \mid a \in B, g \in G\rangle H_i.$$
We begin by showing that $\mathcal{H}_i$ has finite index in $\mathcal{G}$. To this end, consider the homomorphism $\Sigma\colon \mathcal{B} \to B$ defined by $(f)\Sigma=\sum_{x\in K\backslash G}(x)f$. Since every element of $\mathcal{B}$ has finite support and $B$ is abelian, the map $\Sigma$ is well defined. Moreover, observe that $\ker(\Sigma)=[B^K,G]$. Since $\Sigma$ is surjective, it follows that $[\mathcal{B}:[B^K,G]] = |B|$. Consequently,
\[
[\mathcal{G}:\mathcal{H}_i]
=
[B^K:[B^{K},G]]\,[G:H_i]
<
\infty,
\]
showing that $\mathcal{H}_i$ has finite index in $\mathcal{G}$.
Next, we define the homomorphisms $\rho_i:\mathcal{H}_i\to \mathcal{G}$ as extensions of the maps given below
\[
[a^K,g]^{\rho_i}=
\begin{cases}
[a^{K},h^{f_i}], & \text{if } Kg = Kh \ \text{for some} \ h\in H_i,\\
a^{-K}, & \text{otherwise},
\end{cases}
\]
and
\[
h^{\rho_i}=h^{f_i},
\qquad h\in H_i.
\]

Analogously to the proof that $\mathbb{Z}\wr_{K\backslash G}G$ is self-similar, we can use the virtual endomorphisms defined above to show that $B\wr_{K\backslash G}G$ is also self-similar.\\

\subsubsection{Proof of the item (2)} Let $G$ be a non-torsion group satisfying the hypotheses of Theorem $\ref{thmx:ZwrG-SelfSimilar}$. Then, by the previous results, the groups
\[
B\wr_{K\backslash G}G
\quad\text{and}\quad
\mathbb{Z}\wr_{K\backslash G}G
\]
are self-similar. By hypothesis, $A$ is an infinite finitely generated abelian group. Therefore,
\[
A\cong\mathbb{Z}^d\oplus T(A),
\]
where $T(A)$ denotes the torsion subgroup of $A$, which is finite. The proof proceeds by induction, using Proposition 5.2 from \cite{DSS}, which establishes the corresponding result for the normal wreath product. Indeed, the argument employed in that case can be adapted to the present setting, considering $X=K\backslash G$ as the indexing set in place of the indexing group.

\begin{rem}
    If \(X\) is a finite $G_2$-set and \(G_{1}\) and \(G_{2}\) are self-similar,
    it is easy to see that \(G_{1}\wr_X G_{2}\) is self-similar.
\end{rem}


\section{Self-similar projective special linear groups and applications}\label{Section:SelfSimilar-PSL}

In this section, we prove the self-similarity of the projective special linear groups $\PSL(\mathbb{Z}[1/p])$, motivated by a result
due to Šuni\'{c}, presented by Kapovich in \cite[Example 16]{K}. In the case $p=2$, we obtain a strongly self-similar representation. We also present some applications of the results established in the previous section.

\subsection{Proof of Theorem \ref{thmx:PSL-FinitelyPresentedTransitiveSelfSimilar} - Item (1)} In this item, we
consider the group $\PSL(\mathbb{Z}[1/p])$ for $p\geq 3$. We prove that this group is self-similar, with a self-similar representation induced by the virtual endomorphism
\begin{align*}
 f\colon H \coloneqq\left\langle
\begin{pmatrix}
a & 2b\\
c & d
\end{pmatrix}
\right\rangle&\longrightarrow \PSL(\mathbb{Z}[1/p]) \\
 \begin{pmatrix}
a & 2b\\
c & d
\end{pmatrix}&\longmapsto \begin{pmatrix}1/2 & 0 \\ 0 & 1\end{pmatrix}\begin{pmatrix}
a & 2b\\
c & d
\end{pmatrix}\begin{pmatrix} 2 & 0 \\ 0 & 1 \end{pmatrix}=\begin{pmatrix}
a & b\\
2c & d
\end{pmatrix}
\end{align*}
where $H$ is a subgroup of index $3$.   

First observe that every $f$-invariant subgroup of $H$ must be contained in the subgroup of lower triangular matrices. Indeed, under each iteration of $f$, the $(1,2)$-entry is divided by $2$ while remaining divisible by $2$. Hence, after arbitrarily many iterations, this entry must be zero.

On the other hand, the subgroup of lower triangular matrices is not normal. Indeed, it suffices to conjugate it by the matrix
$$\begin{pmatrix} 0 & -1 \\ 1 & 0\end{pmatrix}.$$
Therefore, every normal $f$-invariant subgroup of $H$ is contained in the subgroup of diagonal matrices. However, no subgroup of the group of diagonal matrices is normal. Indeed, this can be seen by conjugating it by the matrix
$$\begin{pmatrix} 1 & -1 \\ 1 & 0\end{pmatrix}.$$
It follows that the $f$-core of $H$ is trivial.

$\hfill\square$

\begin{rem}
For $p=2$, the virtual endomorphism
\begin{align*}
f\colon H\coloneqq \left\langle
\begin{pmatrix}
a & 3b\\
c & d
\end{pmatrix}
\right\rangle
&\longrightarrow \PSL(\mathbb{Z}[1/2])\\
\begin{pmatrix}
a & 3b\\
c & d
\end{pmatrix}&\longmapsto
\begin{pmatrix}
1/3 & 0\\
0 & 1
\end{pmatrix}
\begin{pmatrix}
a & 3b\\
c & d
\end{pmatrix}
\begin{pmatrix}
3 & 0\\
0 & 1
\end{pmatrix}=\begin{pmatrix}
a & b\\
3c & d
\end{pmatrix}
\end{align*}
also induces a self-similar representation of $\PSL(\mathbb{Z}[1/2])$ with degree $4$. However, the parabolic subgroup is non-trivial. 
\end{rem}

\subsection{Proof of Theorem \ref{thmx:PSL-FinitelyPresentedTransitiveSelfSimilar} - Item (2)} Here, we prove that the group $\PSL(\mathbb{Z}[1/2])$ is strongly self-similar, with a self-similar representation induced by the virtual endomorphism
\begin{align*}
 f\colon H\coloneqq \left\langle
\begin{pmatrix}
a & 3b\\
c & d
\end{pmatrix}
\right\rangle&\longrightarrow \PSL(\mathbb{Z}[1/2]) \\ 
 \begin{pmatrix}
a & 3b\\
c & d
\end{pmatrix}&\longmapsto \begin{pmatrix}1/3 & 0 \\ 1/3 & -1\end{pmatrix}\begin{pmatrix}
a & 3b\\
c & d
\end{pmatrix}\begin{pmatrix} 0 & 3 \\ -1 & 1 \end{pmatrix}=\begin{pmatrix}
-b+d & a+b-3c-d\\
-b & a+b
\end{pmatrix}
\end{align*}
where $H$ is a subgroup of index $4$.   
The proof of this result is a consequence of the following lemmas.

\begin{lemma}\label{6}
 Let $(w_n)$ be a non-zero sequence in $\mathbb{Z}[1/2]$ satisfying the recurrence $w_{n+1} = -3w_{n-1}+w_n$. Then there exist $k, M \in \mathbb{N}$ such that $v_3(w_n) \leq M$ for all $n \geq k$.
\end{lemma}
\begin{proof}
Consider the sequences $(A_n)$ and $(B_n)$ defined by
\[
A_{n+1} = A_n - 3A_{n-1}
\quad \text{and} \quad
B_{n+1} = B_n - 3B_{n-1},
\]
with initial conditions $$A_0 = 1,\ \ \ \ \ A_1 = 0 \ \ \ \ \  B_0 = 0, \ \ \ \ \ B_1 = 1$$ 
For any $w_0, w_1$, the solution of the recurrence $w_{n+1} = -3w_{n-1}+w_n$ with these initial values is given by
$$w_n = w_0A_n + w_1B_n,\ \ \ n\geq0$$
Indeed, defining $z_n = w_n - (w_0A_n + w_1B_n)$, one checks immediately that the sequence ($z_n$) satisfies the same homogeneous recurrence and has initial conditions $z_0 = z_1 = 0$. By uniqueness of solutions, it follows that $z_n=0$ for all $n$, that is, $w_n = w_0A_n+w_1B_n$. To understand the behavior of $v_3(w_n)$, we analyze $v_3(A_n)$ and $v_3(B_n)$. For the sequence $(A_n)$, from the recurrence $A_n+1=A_n-3A_{n-1}$, with $A_0 =1$ and $A_1 = 0$, we obtain by induction that
$$A_n\equiv 0 \ (\text{mod}\ 3) \ \text{for all}\  n \geq 1$$
Moreover, $A_2 = -3$, hence $A_2 \equiv 6$ (mod $9$), and again from the recurrence we conclude that
$$A_n \equiv 6\ (\text{mod}\ 9) \ \text{for all}\  n \geq 2$$
In particular, $v_3(A_n)=1$ for all $n\geq 2$. For the sequence $(B_n)$, from the recurrence $B_n+1=B_n-3B_{n-1}$, with $B_0 = 0$ and $B_1 = 1$, it follows that
$$B_n\equiv 1 \ (\text{mod}\ 3) \ \text{for all}\  n \geq 1$$
and therefore $v_3(B_n) = 0$ for all $n\geq 1$.

We now turn to the analysis of $v_3(w_n)$. Write
$$w_0=3^ru, \ \ \ \ \ w_1=3^sv$$
with $u, v \in \mathbb{Z}[1/2]$, such that $3 \nmid u$ and $3 \nmid v$. Then
$$v_3(w_0) = r, \ \ \ \ \ v_3(w_1) = s.$$
For $n \geq 2$, since $v_3(A_n) = 1$ and $v_3(B_n) = 0$, we obtain
$$v_3(w_0A_n) = r+1, \ \ \ \ \ v_3(w_1B_n) = s.$$
We consider two cases\\
\textit{Case 1.} $r+1 \neq s$\\
Without loss of generality, assume $r+1 < s$. Then
$$v_3(w_n) = v_3(w_0A_n + w_1B_n) = \text{min}\{r+1, s\} = r+1 \ \ \text{for all} \ \ n \geq 2.$$
Hence $v_3(w_n)$ is bounded above.\\
\textit{Case 2.} $r+1 = s$\\
Write $t= r+1 = s$. For $n \geq 2$,
$$v_3(w_0A_n) = t, \ \ \ \ \ v_3(w_1B_n) = t $$
and therefore
$$w_n=3^{t}(uA'_n+vB'_n)$$
where
$$A'_n \coloneqq \dfrac{A_n}{3}, \ \ \ \ \ B'_n \coloneqq B_n.$$
Note that $A'_n \in \mathbb{Z}$ for $n \geq 2$, $A'_n \equiv -1$ (mod $3$) and $B'_n \equiv 1$ (mod $3$). Define 
$$z_n\coloneqq uA'_n+vB'_n.$$
Then $v_3(w_n) = t+ v_3(z_n)$. For $n \geq 2$

$$z_n \equiv u(-1)+v(1) \equiv u-v \ (\text{mod}\ 3).$$
We now have two possibilities. If $v-u \not\equiv 0$ (mod $3$), then $z_n\not\equiv0$ (mod $3$) for all $n \geq 2$, and therefore
$$v_3(z_n)=0, \ \ \ \ \ v_3(w_n) = t\ \  \text{for all}\  n \geq 2$$
Thus $v_3(w_n)$ is bounded. If $v-u \equiv 0$ (mod $3$), then $3 \mid z_n$ for all $n \geq 2.$ We may therefore write
$$z_n = 3y_n, \ \ \ n \geq 2 $$
Consequently,
$$w_n=3^{t+1}y_n, \ \ \ n \geq 2.$$
Since $z_{n} = uA'_{n}+vB'_{n}$, using the recurrences for $A'_{n}$ and $B'_{n}$ we obtain
\begin{align*}
    z_{n+1} &= uA'_{n+1}+vB'_{n+1}\\
    & = u(A'_{n}-3A'_{n-1})+v(B'_{n}-3B'_{n-1})\\
    &= uA'_n+vB'_n-3(uA'_{n-1}+vB'_{n-1})\\
    & = z_n-3_{n-1}
\end{align*}
This shows that, after dividing by $3$, the sequence $(y_n)$ also satisfies
$$y_{n+1}=y_n-3y_{n-1}  \ \text{for all}\  n \geq 2$$
By the initial argument, there exists a representation $y_n = y_0A_n+y_1B_n$. Writing $A_n = 3A'_n$ for $n \geq 2$, we obtain $y_n=3y_0A'_n+y_1B'_n$. Defining $u_1 \coloneqq 3y_0$ and $v_1\coloneqq y_1$, it follows that
$$w_n = 3^{t+1}y_n=3^{t+1}(u_1A'_n+v_1B'_n) \ \ \ \ n\geq 2.$$
Since $A'_n \equiv -1$ and $B'_n \equiv 1$ (mod $3$) for all $n \geq 2$, we may repeat exactly the same argument previously applied to $u$ and $v$. If this process did not terminate, we would obtain $3^{t+m} \mid w_n$ for all $m\geq 0$ and $n\geq 2$, which would imply $w_2 = w_3=0$ and hence $w_0=w_1=0$. This completes the proof.

\end{proof}
\begin{lemma}\label{7}
For all integers $k\geq 0$ and $m\in \mathbb{Z}$, the number $2^{2k+2}-11m^2$ is not a perfect square, except in the trivial case $m=0$.
\end{lemma}

Finally, we prove Item (2). The proof will be divided into steps.\\
\textit{Step 1.} Consider
$$P = \begin{pmatrix} 0 & 3 \\ -1 & 1 \end{pmatrix} = \begin{pmatrix} A_1 & B_1 \\ C_1 & D_1 \end{pmatrix},$$
and write $$P^n = \begin{pmatrix}  A_n & B_n \\ C_n & D_n \end{pmatrix}.$$
From the relation $P^{n+1} = P^nP$, we obtain $$P^{n+1} = \begin{pmatrix}  A_n & B_n \\ C_n & D_n \end{pmatrix}\begin{pmatrix} 0 & 3 \\ -1 & 1 \end{pmatrix} =\begin{pmatrix}  -B_n & 3A_n+B_n \\ -D_n & 3C_n+D_n \end{pmatrix} = \begin{pmatrix}  A_{n+1} & B_{n+1} \\ C_{n+1} & D_{n+1} \end{pmatrix}.$$
Thus,
\begin{align}\label{4}
   A_{n+1} = & -B_n, \ \ \ \ \ B_{n+1} = 3A_n +B_n
\end{align}

\vspace{-0.8cm}

\begin{align}\label{5}
   C_{n+1} = & -D_n, \ \ \ \ \ D_{n+1} = 3C_n+D_n.
\end{align}

\hspace{-0.60cm} \textit{Result 1.} For every $n \geq 1$, we have $D_n\equiv 1$ (mod  $3$).\\
 For $n=1$, we have $D_1=1$. Suppose that $D_{n}\equiv 1$ (mod $3$). From relations (\ref{5}) it follows immediately that $D_{n+1} = -3D_{n-1}+D_{n} \equiv 1$ (mod $3$)

\hspace{-0.60cm} \textit{Result 2.} For every $n \geq 1$, we have $B_n\equiv 3$ (mod $9$).\\
For $n=1$, we have $B_1 = 3$. Suppose that $B_n \equiv 3$ (mod $9$). From relations (\ref{4}) we obtain $B_{n+1} = -3B_{n-1} +B_{n} \equiv 3$ (mod $9$).\\

\hspace{-0.6cm} \textit{Step 2.} Let $h \in H$ be such that $h^{f^{n}} \in H$ for all $n \in \mathbb{N}$. We shall prove that $h=e$. Write $$h = \begin{pmatrix}  a & 3b \\ c & d \end{pmatrix} \in H.$$ To conclude that $h=e$, it suffices to prove that $(a-d) = 0$, $b = 0$, and $c = 0$ (since $ad =1 \ \ \text{and} \ \ a=d$). For $n \in \mathbb{N}$, we have $$P^{n} = \begin{pmatrix}  A_n & B_n \\ C_n & D_n \end{pmatrix}, \ \ \ \ P^{-n} = \dfrac{1}{3^n}\begin{pmatrix}  D_n & -B_n \\ -C_n & A_n \end{pmatrix}.$$
Therefore,
\begin{align*}
 h^{f^n} = & P^{-n}hP^{n} = \dfrac{1}{3^n}\begin{pmatrix}  D_n & -B_n \\ -C_n & A_n \end{pmatrix}\begin{pmatrix}  a & 3b \\ c & d \end{pmatrix}\begin{pmatrix}  A_n & B_n \\ C_n & D_n \end{pmatrix}. 
\end{align*}
A direct computation shows that the $(1,2)$-entry of $h^{f^n}$ equals
$$\dfrac{1}{3^n}\left(B_nD_n(a-d)+3bD_n^{2}-cB_n^{2}\right).$$
Since $h^{f^n} \in H$ for all $n$, this entry must lie in $3\mathbb{Z}[1/2]$. Hence its multiplicity of $3$ must increase with $n$.

Define $$b_n \coloneqq \dfrac{B_n}{3}, \ \ \ \ \ \ d_n \coloneqq D_n.$$ 
By Results 1 and 2, we have $b_n, d_n \in \mathbb{Z}$ and $b_n, d_n \equiv 1$ (mod $3$). Thus the $(1,2)$-entry can be written as
\begin{align*}
 \dfrac{1}{3^{n-1}}\left((a-d)b_nd_n+bd_n^{2}-3cb_n^{2}\right).
\end{align*}
Define $$\Sigma_n = (a-d)b_nd_n+bd_n^{2}-3cb_n^2.$$
Then $v_3(\Sigma_n) \geq n$ for all $n$. Setting $$X \coloneqq (a-d), \ \ \ \ \  Y \coloneqq b,\ \ \ \ \ Z \coloneqq c,$$ we obtain $ \Sigma_n =  Xb_nd_n+Yd_n^{2}-3Zb_n^2$. From this point on, the condition $v_3(\Sigma_n) \geq n$ for all $n$ will necessarily imply that $X=Y=Z = 0$. We now show how this occurs. Adding and subtracting $Yb_nd_n$ and $3Yb_n^2$, we obtain
\begin{align*}
    \Sigma_n & = Xb_nd_n+Yd_n^{2}-3Zb_n^2\\ &= Xb_nd_n+Yd_n^{2}-3Zb_n^2+Yb_nd_n-Yb_nd_n+3Yb_n^{2}-3Yb_n^{2}\\
    & = Y(d_n^2-b_nd_n+3b_n^2)+(X+Y)b_nd_n-3(Z+Y)b_n^2.
\end{align*}
Rearranging the terms, we obtain
\begin{align*}
    \Sigma_n & = Y(d_n^2-b_nd_n+3b_n^2)+(X+Y)b_nd_n-3(Z+Y)b_n^2.
\end{align*}
It is known that $d_n^2-b_nd_n+3b_n^2 = 3^n$ for all $n \in \mathbb{N}$, since $b_n$ and $d_n$ satisfy the same recurrence relations as $B_n$ and $D_n$, respectively, with initial conditions $b_1 = 1, b_2 = 1$ and $d_1 = 1, d_2 = -2$, and the general solutions are $b_n = \frac{2(\sqrt{3})^n}{\sqrt{11}}\sin(n\theta)$ and $d_n = (\sqrt{3})^n\left(\cos(n\theta)+\frac{1}{\sqrt{11}}\sin(n\theta)\right)$. Therefore,
$$\Sigma_n = Y3^n+(X+Y)b_nd_n-3(Z+Y)b_n^2.$$
Since $v_3(\Sigma_n) \geq n$ for all $n$, it follows that
$$(X+Y)b_nd_n-3(Z+Y)b_n^2 \equiv 0\ (\text{mod}\ 3^n) \ \text{for all} \ n.$$
Define
$$w_n = (X+Y)b_nd_n-3(Z+Y)b_n^2$$
Using the recurrences
$$b_{n+1} = -3b_{n- 1}+b_n, \ \ \ \ \ \ d_{n+1} = -3d_{n-1}+d_n$$
one verifies by direct substitution that $w_{n+1} = -3w_{n-1}+w_n$. By Lemma \ref{6}, it follows that $w_0=w_1 = 0$. From the definition of $w_n$ we conclude that $X+Y=0$ and $Z+Y=0$. Thus,
$$X=-Y=Z, \ \ \text{that is,} \ \ a-d=-b=c.$$
Setting $t\coloneqq a-d=c$, we obtain $b=-t$ and $d=a-t$. Since $h \in H$, we have $ad-3bc=1$, and therefore $a(a-t)-3(-t)t = 1$, that is, $a^2-at+3t^2-1 =0$. Solving this quadratic equation in $a$, we obtain
$$a = \dfrac{t\pm\sqrt{4-11t^2}}{2}$$
Since $a \in \mathbb{Z}[1/2]$, we may write $t=\dfrac{m}{2^k}$ with $m\in \mathbb{Z}$ and $k \in \mathbb{N}$. Substituting, we obtain
$$a = \dfrac{m\pm\sqrt{2^{2k+2}-11m^2}}{2^{k+1}}$$
For $a \in \mathbb{Z}[1/2]$, it is necessary that $2^{2k+2}-11m^2$ be a perfect square. However, in Lemma \ref{7} we showed that this occurs only when $m=0$. Consequently, $t = 0$, as desired. This completes the proof.

\(\hfill\square\)

Before we prove the next item, we prove the following corollary of
Item (2).
\begin{corollary}
Let $p\geq3$ be a prime number and let $G = \PSL(\mathbb{Z}[1/p])$. If $K$ is the subgroup of $G$ consisting of lower triangular matrices, then $A\wr_{K\backslash G} \PSL(\mathbb{Z}[1/p])$ is self-similar, where $A$ is a finitely generated abelian group.
\end{corollary}
\begin{proof}
As shown in the proof of Theorem \ref{thmx:PSL-FinitelyPresentedTransitiveSelfSimilar} Item (1), $G$ admits a self-similar representation induced by the virtual endomorphism $f$. Observe that $K$ coincides with the parabolic subgroup associated with $f$, that is,
$K = H_\omega$ and has infinite index in $\PSL(\mathbb{Z}[1/p])$. Naturally, $H_\omega \leq K.$ By construction,
$$(K\cap H)^f=K^f\leq K.$$
It remains to verify the final item of the condition $(\Omega)$. By the construction of $f$, we have $h^f\in K$ only if $h\in K$. Consequently, by Theorem \ref{thmx:ZwrG-SelfSimilar}, the group $A\wr_{K\backslash G} \PSL(\mathbb{Z}[1/p])$ is self-similar.
\end{proof}

\subsection{Proof of Theorem \ref{thmx:PSL-FinitelyPresentedTransitiveSelfSimilar} - Item (3)} We recall that 
this item considers the group $G = \PSL(\mathbb{Z}[1/2])$ and $K = \PSL(\mathbb{Z})$ the subgroup of $G$. We need to prove that 
$A \wr_{K\backslash G} \PSL(\mathbb{Z}[1/2])$ is self-similar. 

Note that $G$ admits a self-similar representation induced by the virtual endomorphism $f$ constructed in the proof of Theorem \ref{thmx:PSL-FinitelyPresentedTransitiveSelfSimilar} Item (2), and $K$ has infinite index in $\PSL(\mathbb{Z}[1/2])$. We verify that the condition $(\Omega)$ is satisfied.

\begin{enumerate}
    \item[(i)] By Item (2), the parabolic subgroup satisfies
    $H_{\omega}=e$.
    Hence,
    \[
    H_{\omega}\leq K.
    \]

    \item[(ii)] We show that $(H\cap K)^f\leq K$. Indeed, let
    \[
    h=
    \begin{pmatrix}
    a & 3b\\
    c & d
    \end{pmatrix}
    \in H\cap K.
    \]
    Then $a,3b,c,d\in\mathbb{Z}$. Since $3b\in\mathbb{Z}$ and $b\in\mathbb{Z}[1/2]$, it follows that $b\in\mathbb{Z}$. Now,
    \[
    h^f=
    \begin{pmatrix}
    -b+d & a+b-3c-d\\
    c & a+b
    \end{pmatrix}.
    \]
    Using the observation above, all entries of $h^f$ belong to $\mathbb{Z}$, hence $h^f\in K$.

    \item[(iii)] Let $h=
    \begin{pmatrix}
    a & 3b\\
    c & d
    \end{pmatrix}
    \in H$,
    and suppose that $h^{f}\in K$.
    That is,
    \[
    \begin{pmatrix}
    -b+d & a+b-3c-d\\
    c & a+b
    \end{pmatrix}
    \in K.
    \]
    Therefore, $b, \ a+b,\ -b+d, \ a+b-3c-d\in\mathbb{Z}$.
    It follows that $a, b, d\in\mathbb{Z}$
    and $3c\in\mathbb{Z}$.
    Since $c\in\mathbb{Z}[1/2]$, we conclude that $c\in\mathbb{Z}$.
    Hence, $h\in K$.
\end{enumerate}
Thus all items are satisfied. Therefore, by Theorem \ref{thmx:ZwrG-SelfSimilar}, the group $A \wr_{K\backslash G} \PSL(\mathbb{Z}[1/2])$ is self-similar.

\(\hfill\square\)

\begin{corollary} \label{cor4.2} The groups $\mathbb{Z} \wr \mathbb{Z}^{(\omega)}$ and $\mathbb{Z} \wr \mathbb{Z}^{\omega}$ are self-similar.
\end{corollary}

\begin{proof}
Consider
$$G = \mathbb{Z}^{\omega} = \{(n_1, n_2, n_3, n_4, n_5, n_6, ...) \mid n_i \in \mathbb{Z}\},$$
$$H = \{(2n_1, n_2, n_3, n_4, n_5, n_6, ...) \mid n_i \in \mathbb{Z}\},$$
and the homomorphisms $f_1: H \rightarrow G$ and $f_2: G \rightarrow G$ that extend, respectively, the maps
\[f_{1}\colon (2n_1, n_2, n_3, n_4, ..., n_{2i - 1}, n_{2i}, \dots) \mapsto (n_2, n_1, n_4, n_3, ..., n_{2i}, n_{2i - 1}, \dots)\]
and
\[f_{2}\colon (n_1, n_2, n_3, \dots, n_{3i - 2}, n_{3i - 1}, n_{3i}, \dots) \mapsto (n_2, n_3, n_1, \dots, n_{3i - 1}, n_{3i}, n_{3i-2}, \dots).\]
Note that the group $G$ is self-similar, induced by the virtual endomorphisms $f_1$ and $f_2$. Moreover, $f_1$ and $f_2$ are monomorphisms and the parabolic subgroup $H_{\omega}$ is trivial. By Theorem \ref{thmx:ZwrG-SelfSimilar}, the group $\mathbb{Z} \wr_{H_{\omega} \setminus G} G = \mathbb{Z} \wr \mathbb{Z}^{\omega}$ is self-similar. 

For the first part, the group $\mathbb{Z}^{(\omega)}$ is $\{f_1, f_2\}$-semi-invariant, hence $\mathbb{Z} \wr \mathbb{Z}^{(\omega)}$ is self-similar.

\end{proof}

\begin{corollary}\label{cor:Contratil}
    Let $G$ be a non-torsion group with a self-similar representation induced by the virtual endomorphisms $f_i\colon H_i\to G$ for $1 \leq i \leq s$. If $K = H_\omega = e$ and $\text{Ker}\,(f_i) = e$ for every $i$, then $A\wr G$ is self-similar, for every finitely generated abelian group \(A\).
\end{corollary}

\subsection{Proof of Theorem \ref{thm:BGcontratil}} \label{Subsec4.1} Recall that
Theorem \ref{thm:BGcontratil} establishes that, under certain 
conditions, when $B$ is a finite abelian group, the permutational 
wreath product $B\wr_{K\setminus G} G$ is a contracting self-
similar group. In order to prove it, we need the following lemma.

\begin{lemma}\label{Sndeproduto}
Let $G$ be a self-similar group. Then, for any $a,b\in G$ and every $n\geq 0$,
\[
S_n(ab)\subseteq S_n(a)\,S_n(b).
\]
\end{lemma}
\begin{proof}
Let $f_i\colon H_i\longrightarrow G$, for  $1\leq i\leq s$, be the virtual endomorphisms inducing the self-similar representation of $G$. The proof proceeds by induction on $n$. For $n=0$, by the definition $S_0(g)=\{g\}$ for every $g\in G$, we have
\[
S_0(ab)=\{ab\}\subseteq \{a\}\{b\}=S_0(a)S_0(b).
\]
Now assume that, for some $n\geq 0$, $S_n(ab)\subseteq S_n(a)S_n(b)$. We will show that
\[
S_{n+1}(ab)\subseteq S_{n+1}(a)S_{n+1}(b).
\]
Let $z\in S_{n+1}(ab)$. By the definition of $S_{n+1}(ab)$, there exist $x\in S_n(ab)$ and a transversal representative $t_{ij}\in T_i$
such that
\[
z=\theta_i(x,t_{ij})^{f_i}.
\]
By the induction hypothesis, $x=uv$,
where $u\in S_n(a)$ and $v\in S_n(b)$. Recall that
\[
\theta_i(g,t_{ij})
=
t_{ij}\,g\,t_{ij^{\sigma_i(g)}}^{-1},
\]
where $\sigma_i(g)$ denotes the permutation induced by $g$ on the transversal representatives. Thus, $\theta_i(x,t_{ij}) = t_{ij}\,x\,t_{ij^{\sigma_i(x)}}^{-1}$. Since $x=uv$, it follows that $j^{\sigma_i(x)}=j^{\sigma_i(u)\sigma_i(v)}$,
and therefore,
\[
\begin{aligned}
\theta_i(x,t_{ij})
&=
t_{ij}\,uv\,t_{ij^{\sigma_i(u)\sigma_i(v)}}^{-1} \\
&=
(t_{ij}u\,t_{ij^{\sigma_i(u)}}^{-1})
(t_{ij^{\sigma_i(u)}}v\,t_{ij^{\sigma_i(u)\sigma(v)}}^{-1}).
\end{aligned}
\]
Observing that $\theta_i(u,t_{ij}) = t_{ij}u\,t_{ij^{\sigma_i(u)}}^{-1}$ and $\theta_i\!(v,t_{i,j^{\sigma_i(u)}}) = t_{ij^{\sigma_i(u)}}v\,t_{ij^{\sigma_i(u)\sigma_i(v)}}^{-1}$,
we obtain
\[
\theta_i(x,t_{ij})
=
\theta_i(u,t_{ij})\,
\theta_i\!(v,t_{ij^{\sigma_i(u)}}).
\]
Applying $f_i$, we conclude that
\[
z
=
\theta_i(u,t_{ij})^{f_i}\,
\theta_i\!(v,t_{ij^{\sigma_i(u)}})^{f_i}.
\]

By the definition of $S_{n+1}(a)$ and $S_{n+1}(b)$, we have 
$\theta_i(u,t_{ij})^{f_i}\in S_{n+1}(a)$ and $\theta_{i}(v,t_{i,j^{\sigma(u)}})^{f_i}\in S_{n+1}(b)$. Hence, $z\in S_{n+1}(a)S_{n+1}(b)$. Since $z$ was chosen arbitrarily in $S_{n+1}(ab)$, it follows that
\[
S_{n+1}(ab)\subseteq S_{n+1}(a)S_{n+1}(b).
\]
\end{proof}

By Theorem \ref{thmx:ZwrG-SelfSimilar}, it remains only to prove contractivity. Let $T_i = \{t_{i1}, t_{i2},...,t_{im_i}\}$ be the transversal of $H_i$ in $G$ given in the hypothesis. Since $B$ is finite, write $B=\{a_1,\ldots,a_n\}$. For each $i$, the set
\[
T_{\mathcal{H}_i}
=
\left\{
a_r^{K}t_{ij}
\;\middle|\;
1\leq r\leq n,\;
1\leq j\leq m_i
\right\},
\]
is a transversal of $\mathcal{H}_i$ in $\mathcal G$. We shall prove that, for every element $y \in \mathcal{G}$, there exists a finite subset $S_{\mathcal{G}}\subset \mathcal{G}$ and an integer $N\in\mathbb{N}$ such that
\[
\overline{S}_n(y)\subset S_{\mathcal{G}},
~\forall~n\ge N.
\]

Fix $i\in\{1,\ldots,s\}$. To simplify the notation, we omit the index $i$ from the transversal. Let $\theta:G\times T_H\longrightarrow H$ denote the Schreier map associated with the original self-similar representation of $G$, and let $\overline{\theta}:\mathcal{G}\times T_{\mathcal{H}}\longrightarrow \mathcal{H}$ denote the Schreier map associated with the self-similar representation of $\mathcal{G}$. Observe that these are distinct maps. The strategy of the proof is to describe the values of $\overline{\theta}$ in terms of the values of $\theta$, thereby allowing us to transfer the contracting property of $G$ to $\mathcal{G}$. We begin by dividing the proof into two main cases, from which the general case follows immediately.

\medskip \hspace{-0.45cm}\textit{Case $1$}: First, consider an element of the base subgroup of $\mathcal{G}$ of the form
\[
w_1=a_1^{Kg_1}\cdots a_l^{Kg_l}.
\]
We compute
\[
\overline{\theta}
\!\left(
w_1,
\,a_r^{K}t_j
\right).
\]
By the definition of the Schreier map associated with $\mathcal{G}$,
$$\overline{\theta}
\!\left(
w_1,
\,a_r^{K}t_j
\right) = a_r^{K}t_jw_1t_j^{-1}a_d^{-K}$$
where $a_d$ is determined by the condition
\[
\mathcal{H}\,a_r^{K}t_jw_1
=
\mathcal{H}\,a_d^{K}t_j
\iff
a_d=a_ra_1\cdots a_l.
\]
Hence,
\begin{align*}
\overline{\theta}
\!\left(
w_1,
\,a_r^{K}t_j
\right)
&=
a_r^{K}t_j\,
a_1^{Kg_1}\cdots a_l^{Kg_l}\,
t_j^{-1}a_d^{-K}
\\
&=
t_j\,
a_1^{Kg_1}\cdots
a_l^{Kg_l}\,
t_j^{-1}
a_1^{-K}\cdots a_l^{-K}.
\end{align*}
By inserting intermediate factors of the form $t_{j^{\sigma(g_\lambda)}}^{-1}t_{j^{\sigma(g_\lambda)}}$ for $1 \leq \lambda \leq l$, together with the factor $t^{-1}_jt_j$, we obtain
\begin{align*}
   = &\ a_1^{K t^{-1}_{j^{\sigma(g_1)}}
\theta(g_1,t_j)^{-1}} a_2^{K t^{-1}_{j\sigma(g_2)}
\theta(g_2,t_j)^{-1}}
\cdots a_l^{K t^{-1}_{j\sigma(g_l)}
\theta(g_l,t_j)^{-1}}
a_1^{-K}\cdots a_l^{-K}\\
= & [
a_1^{K}, \,t^{-1}_{j^{\sigma(g_1)}}\theta(g_1,t_j)^{-1}]
[
a_2^{K}, \, t^{-1}_{j^{\sigma(g_2)}}\theta(g_2,t_j)^{-1}]
\cdots
[
a_l^{K}, \, t^{-1}_{j^{\sigma(g_l)}}\theta(g_l,t_j)^{-1}
]. 
\end{align*}
Applying $\rho$, without loss of generality, assume that

\[
\overline{\theta}\!\left(
w_1,
a_r^{K}t_j
\right)^{\rho}
=
\begin{cases}
\left(a_1...a_l\right)^{-K},
&
\hspace{-1cm} \text{if }
t_{j^{\sigma(g_\lambda)}} \notin KH, \forall \lambda,
\\[2mm]
[a_1^K, \left(\theta(g_1,t_j)^f\right)^{-1}]
\cdots
[a_l^K, \left(\theta(g_l,t_j)^f\right)^{-1}],
&
\text{if }
t_{j^{\sigma(g_\lambda)}} \in KH, \forall \lambda.
\end{cases}
\]
The cases in which the elements $t_{j^{\sigma(g_\lambda)}}$, for $1 \leq \lambda \leq l$, as well as the intermediate cases in which some of these elements belong to $KH$ while others do not, are treated analogously. 

Writing $\theta(g_1,t_j)^f=x_1, \ldots, \theta(g_l,t_j)^f=x_l$, it follows that
\[
\overline{\theta}\!\left(
w_1,
a_r^{K}t_j
\right)^{\rho}
=
\begin{cases}
\left(a_1...a_l\right)^{-K},
&
\text{if }
t_{j^{\sigma(g_\lambda)}} \notin KH, \forall \lambda,
\\
a_1^{Kx_1^{-1}}a_2^{Kx_2^{-1}}\cdots a_l^{Kx_l^{-1}}
a_1^{-K}\cdots a_l^{-K}
=
w_2,
&
\text{if }
t_{j^{\sigma(g_\lambda)}} \in KH, \forall \lambda.
\end{cases}
\]
Applying the same procedure to $w_2$ and using the properties of the Schreier map $\theta$, we conclude that
\begin{align*}
\overline{\theta}\!\left(
w_2,
a_r^{K}t_k
\right)^{\rho}
=&
\begin{cases}
\left(a_1...a_l\right)^{-K},
&
\text{if }
t_{k^{\sigma(g_\lambda)}} \notin KH, \forall \lambda,
\\[2mm]
[a_1^K, \left(\theta(x_1^{-1},t_k)^{-1}\right)^f]
\cdots
[a_l^K, \,\left(\theta(x_l^{-1},t_k)^{-1}\right)^f],
&
\text{if }
t_{k^{\sigma(g_\lambda)}} \in KH, \forall \lambda.
\end{cases}
\\
=&
\begin{cases}
\left(a_1...a_l\right)^{-K},
&
\text{if }
t_{k^{\sigma(g_\lambda)}} \notin KH, \forall \lambda,
\\[2mm]
[a_1^K, \,\theta(x_1,t_{k^{\sigma(x_1^{-1})}})^f]
\cdots
[a_l^K, \,\theta(x_l,t_{k^{\sigma(x_l^{-1})}})^f],
&
\text{if }
t_{k^{\sigma(g_\lambda)}} \in KH, \forall \lambda.
\end{cases}
\end{align*}
Iterating this procedure, we see that $\overline{S}_n(w_1)$ depends only on the sets $S_n(g_1),..., S_n(g_l)$.

\medskip Since $G$ is contracting, there exist a finite subset $S_G\subset G$ and integers $n_1,\ldots,n_l\in\mathbb{N}$ such that
\[
S_n(g_\lambda)\subset S_G,
~\forall~n>n_\lambda,
\]
where $\lambda \in \{1, ..., l\}$. Setting
\[
N_1=\text{max}\{n_1,\ldots,n_l\}.
\]
we obtain, $S_n(g_\lambda) \subset S_G$, for every $\lambda \in \{1, \dots, l\}$ and every $n>N_1$. Hence, we conclude that
\[
\overline{S}_n(w_1)\subset B^{S_G},
~\forall~n>N_1.
\]

\medskip \hspace{-0.45cm}\textit{Case $2$}: Now consider an element of the top subgroup $G$ embedded in $\mathcal{G}$ of the form $$1_{\mathcal{B}}g \  \text{with}  \ g \in G.$$ 
We compute $\overline{\theta}
\!\left(
g,
\,a_r^{K}t_j
\right)$. By the definition of $\overline{\theta}$, we obtain
$$\overline{\theta}
\!\left(
g,
\,a_r^{K}t_j
\right) = a_r^{K}t_jgt_{j^{\sigma(g)}}^{-1}a_r^{-K}$$
where the above equality is determined by the condition 
\[
\mathcal{H}\,a_r^{K}t_jg
=
\mathcal{H}\,a_d^{K}t_i
\iff
a_d= a_r \ \ \text{e} \ \ i = j^{\sigma(g)}.
\]
Hence,
\begin{align*}
\overline{\theta}(g,a_r^{K}t_j)
&=
a_r^{K}t_jg\,t_{j^{\sigma(g)}}^{-1}a_r^{-K}
\\
&=
a_r^{K}\theta(g,t_j)a_r^{-K}
\\
&=
\theta(g,t_j)\,[a_r^{K},\theta(g,t_j)].
\end{align*}
Applying $\rho$, we obtain
\[
\overline{\theta}(g,a_r^{K}t_j)^{\rho}
=
\theta(g,t_j)^{f}\,[a_r^{K},\theta(g,t_j)^{f}].
\]
Writing $x=\theta(g,t_j)^{f}$, we have
\[
\overline{\theta}(g,a_r^{K}t_j)^{\rho}
=
x\,[a_r^{K},x]
=
a_r^{K}xa_r^{-K}.
\]

Applying $\overline{\theta}$, once again, we obtain
\[\overline{\theta}(a_r^{K}xa_r^{-K},a_d^{K}t_k) = a_d^{K}t_k
a_r^{K}xa_r^{-K}
t_{k^{\sigma(x)}}^{-1}
a_d^{-K}\]
where the above equality is determined by the condition
\[
H\,a_d^{K}t_k\,a_r^{K}xa_r^{-K}
=
H\,a_l^{K}t_i
\iff
a_l= a_d
\quad\text{and}\quad
i={k^{\sigma(x)}}.
\]
Hence,
\begin{align*}
\overline{\theta}(a_r^{K}xa_r^{-K},a_d^{K}t_k)
&=
a_d^{K}t_k
a_r^{K}xa_r^{-K}
t_{k^{\sigma(x)}}^{-1}
a_d^{-K}
\\
&=
a_d^{K}a_r^{Kt_k^{-1}}
t_kx
t_{k^{\sigma(x)}}^{-1}
a_r^{-K t^{-1}_{k\sigma(x)}}
a_d^{-K}
\\
&=
a_d^{K}a_r^{Kt_k^{-1}}
\theta(x,t_k)
a_r^{-K t^{-1}_{k^{\sigma(x)}}}
a_d^{-K}
\\
&=
\theta(x,t_k)\,a_d^{K\theta(x,t_k)}a_r^{K t^{-1}_k\theta(x,t_k)}
a_r^{{-K t^{-1}_{k^{\sigma(x)}}}}\,
a_d^{-K}
\\
&=
\theta(x,t_k)[a_d^{K}, \theta(x,t_k)]
[a_r^{K}, t^{-1}_{k}\theta(x,t_k)]
[a_r^{-K}, t^{-1}_{k^{\sigma(x)}}]
[a_d^{-K}, 1]
.
\end{align*}
Applying $\rho$, observe that
\[
\overline{\theta}(a_r^{K}xa_r^{-K},a_d^{K}t_k)^{\rho}
=
\begin{cases}
\theta(x,t_k)^{f}\,[a_d^{K},\theta(x,t_k)^{f}],
&
\text{if } t_k\notin KH
\text{ and }
t_{k^{\sigma(x)}}\notin KH,
\\[2mm]
\theta(x,t_k)^{f}\,[a_d^{K}, \theta(x,t_k)^{f}]a_r^{-K},
&
\text{if } t_k\notin KH
\text{ and }
t_{k^{\sigma(x)}} \in KH,
\\[2mm]
\theta(x,t_k)^{f}\,[(a_da_r)^{K}, \theta(x,t_k)^{f}]a_r^{K},
&
\text{if } t_k\in KH
\text{ and }
t_{k^{\sigma(x)}} \notin KH,
\\[2mm]
\theta(x,t_k)^{f}\,[(a_da_r)^{K}, \theta(x,t_k)^{f}],
&
\text{if } t_k\in KH
\text{ and }
t_{k^{\sigma(x)}} \in KH.
\end{cases}
\]
Iterating this procedure, we see that $\overline{S}_n(g)$ depends only on $S_n(g)$. 

\medskip Since $G$ is contracting, there exists $N_2$ such that $S_n(g) \subseteq S_{G}$ for all $n \geq N_2$. Hence,
$$\overline{S}_n(g) \subset {S_G}^{B}B^{S_G} \qquad \forall\, n>N_2.$$

Finally, let
$$y=a_1^{Kg_1}\dots a_l^{Kg_l}g $$
be an arbitrary element of $\mathcal{G}$. By Lemma \ref{Sndeproduto}, the following inclusion holds:
$$\overline{S}_n(y) \subseteq \overline{S}_n(a_1^{Kg_1}\dots a_l^{Kg_l})\overline{S}_n(g).$$
The conclusions of Cases $1$ and $2$ imply that there exist integers $N_1, N_2 \in \mathbb{N}$ such that 
$$\overline{S}_n(a_1^{Kg_1}\dots a_l^{Kg_l}) \subset B^{S_G},$$
and
$$\overline{S}_n(g) \subset {S_G}^{B}B^{S_G},$$
for all sufficiently large $n$. Setting $N=\text{max}\{N_1, N_2\}$, we obtain
$$\overline{S}_n(y) \subseteq B^{S_G}{S_G}^{B}B^{S_G}$$
for every $n \geq N$. Taking $S_\mathcal{G} = B^{S_G}{S_G}^{B}B^{S_G}$, we conclude that $\mathcal{G}$ is contracting.

\(\hfill\square\)

\begin{ex} \label{4.8}
    The group \(C_{2}\wr(C_{2}\wr \mathbb{Z})\) is a contracting self-similar group. 
    
    Consider
    $$C_{2}\wr(C_{2}\wr \mathbb{Z}) \simeq \langle a \rangle \wr (\langle b \rangle \wr \langle x \rangle ).$$
    Let $G$ be the group $\langle b \rangle \wr \langle x \rangle$. So $H = \langle [b, x^2] \rangle^{^{\langle x \rangle}} \langle x^2 \rangle$ has finite index in $G$. The homomorphism $f: H \rightarrow G$ that extends the map $[b, x^2] \mapsto [b, x]$, $x^2 \mapsto x$ is a simple virtual endomorphism that induces a contracting strongly self-similar representation of $G$. As $f$ is a monomorphism, by Theorem \ref{thmx:ZwrG-SelfSimilar} and \ref{thm:BGcontratil}, the group $\langle a \rangle \wr G$ is a contracting self-similar group.
\end{ex}
\section{Proof of Theorem \ref{thmx:SRN-FinitePresentation}}\label{Section:ProvaUltimosTeoremas}

Let \(G\leq\Am\) be a non-torsion contracting self-similar group satisfying
the hypotheses of Theorem \ref{thmx:ZwrG-SelfSimilar}.
We cannot guarantee that the group \(\mathbb{Z}^{d} \wr_X G\), where \(X=K\setminus G\), is finitely presented
(see \cite[Theorem~1.1]{Cornulier2006}). However, we prove that \(V_{d(m+1)}(\mathbb{Z}^{d}\wr_{X} G)\) is
finitely presented. We also prove that we can embed the permutational wreath product
\(\mathbb{Z}^{d}\wr_X G\) in a finitely presented simple group. Part of the results in this section
are drawn from \cite[Section 9]{Nek2004} and \cite[Subsection 5.5]{Nek2017} and we
reproduce them here with explicit references for convenience of the reader.

The idea is to define a set which is an ``extension'' of the nucleus of \(G\).
We call this set the \emph{extended nucleus of \(\mathbb{Z}^{d}\wr_{X} G\)}.
It is sufficient to prove both parts of Theorem
\ref{thmx:SRN-FinitePresentation} for the group \(\mathbb{Z}\wr_X G\), since, by Theorem \ref{thmx:ZwrG-SelfSimilar} Item 2, the group \(\mathbb{Z}^d\wr_X G\) is self-similar of degree \(d(m+1)\). This is
convenience, since it simplify the arguments. In the end of Subsubsections
\ref{subsubsec:finitepresentation} and
\ref{subsubsec:finiteabelizanization}, we explain the necessary adaptation.

\subsection{Setup and notation}
\subsubsection{Self-similar representation}
Let \(\mathbb{Z} \wr_X G = \langle a \rangle \wr_X G\). Consider the virtual endomorphisms \(f_i\colon H_i \to G\), for $1 \leq i \leq s$, and, for each i, a right transversal
\[T_i = \{e, t_{i2}, \dots, t_{im_i}\}\] for $H_i$ in $G$, which together induce the contracting self-similar representation of \(G\). By Theorem \ref{thmx:ZwrG-SelfSimilar}, each virtual endomorphisms \(f_i\), can be extended to a virtual endomorphisms \(\rho_i\colon \mathcal{H}_i \to \mathbb{Z} \wr_X G\). Together with the virtual endomorphism \(\rho\), these virtual endomorphisms induce a self-similar representation of \(\mathbb{Z} \wr_X G\) on $\mathcal{A}_{m+1}$.

By Section $2.2$, the action of the generator \(a\) on the first level of the tree \(\mathcal{A}_{m+1}\) is given by
\[a^\varphi = \left( (t_{11} a t_{11}^{-1})^{\rho_1\varphi}, \dots, (t_{1m_1} a t_{1m_1}^{-1})^{\rho_1\varphi}, \dots, (t_{s1} a t_{s1}^{-1})^{\rho_s\varphi}, \dots, (t_{sm_s} a t_{sm_s}^{-1})^{\rho_s\varphi}, x^{\rho \varphi} \right)\]
since $\mathcal{H}_it_{ij}a = \mathcal{H}_i{a}^{t_{ij}^{-1}}t_{ij} = \mathcal{H}_it_{ij}$. The construction of the maps \(\rho_i, \rho \) yields
\[a^\varphi=(\underbrace{a^{\varphi}, e, \dots, e}_{m_{1}}, \dots, \underbrace{a^{\varphi}, e, \dots, e}_{m_{s}}, x^{\varphi}).\]
Setting $a^\varphi = \gamma$ and $\chi = x^\varphi$ we may therefore write
\[\gamma = (\underbrace{\gamma, e, \dots, e}_{m_1}, \dots, \underbrace{\gamma, e, \dots, e}_{m_s}, \chi).\]

\subsubsection{Some facts about SRN-groups}
Let us recall a standard fact about the Higman--Thompson group \(V_m\).

\begin{lemma}\cite[Lemma 9.12]{Nek2004}\label{Lema3.1-Nekra}
 Let \(\mathcal{C}_{1}=\{C_{v_{1}},\ldots,C_{v_{n}}\}\) and \(\mathcal{C}_{2}=\{C_{u_{1}},\ldots,C_{u_{n}}\}\) be finite partitions of \(\C_{m}\) into disjoint cones.
 Let \(\alpha\colon \mathcal{C}_{1}\to \mathcal{C}_{2}\) a bijection. Then there exists
 \(g\in V_{m}\) such that \((vw)g=(v)^{\alpha}w\), for all \(v\in \mathcal{C}_{1}\)
 and \(w\in Y^{\omega}\).
\end{lemma}

Next, we recall the following results, due to Nekrashevych.
\begin{theorem}\cite[Theorem 4.7]{Nek2017}\label{thm:Nekra-SimpleCommutator}
 For every self-similar group \(G\leq\Am\), the commutator subgroup \([V_{m}(G),V_{m}(G)]\) is simple.
\end{theorem}
\begin{theorem}\cite[Theorem 5.9]{Nek2017}\label{thm:Nekra-FinitelyPresentation}
 For every contracting self-similar group \(G\leq\Am\), \(V_{m}(G)\) is finitely presented.
\end{theorem}

Since \(G \leq \Am\) is a contracting self-similar group, it admits a nucleus \(\N(G)\). We define the
\emph{extended nucleus} of \(\ZwrG\) as the following finite set
\[\N\coloneqq \N(G)\cup Q(\gamma),\]
where \(\gamma\in \mathcal{A}_{m+1}\) represents the generator of the \(\mathbb{Z}\)-factor as an automorphism of \(\T_{m+1}\). This set is finite, since $G$ is a contracting self-similar group. Then, we may assume \(\N\) as a
generating set of \(\ZwrG\), since replacing \(\ZwrG\) by \(\langle\N\rangle\) does not alter \(V_{m+1}(\ZwrG)\).

For \(v\in Y^{*}\) and \(g\in V_{m+1}(\ZwrG)\), let us define
\[(w)L_{v}(g)=\begin{cases}
 v(u)^{g},~\text{if}~w=vu~\text{for some}~u\in \widetilde{Y}^{\omega} \\ w,~\text{otherwise}
\end{cases},\]
\noindent where \(\widetilde{Y}\coloneqq \{1,\ldots,m+1\}\)\footnote{We recall that the action of \(G\) can be extended to the \((m+1)\)-ary tree by adjoining a fixed branch.}.
\begin{rem}
 In case of \(v=1\in \Y\), for simplicity, throughout this paper, we
 write \(L\) in place of \(L_{1}\), that is, \(L(g)\coloneqq L_{1}(g)\), with \(g\in\N\).
\end{rem}
We recall the following proposition from \cite{Nek2017}.
\begin{proposition}\cite[Proposition 5.10]{Nek2017}\label{prop:Propriedades-L}
 \begin{enumerate}
 \item For every \(v\in \Y^{*}\), the map \(L_{v}\colon V_{m+1}(\ZwrG)\to V_{m+1}(\ZwrG)\)
is a group monomorphism.
 \item The subgroups \(L_{v}(V_{m+1}(\ZwrG)),L_{u}(V_{m+1}(\ZwrG))\le V_{m+1}(\ZwrG)\) commute if \(v,u\in \Y^{*}\) are not comparable.
 \item If \(v,u\in \Y^{*}\) are non-empty, then \(f^{-1}L_{v}(g)f=L_{u}(g)\),
for all \(g\in V_{m+1}(\ZwrG)\), where \(f\in V_{m+1}(\ZwrG)\) satisfies \((vw)f=uw\), for
every \(w\in \Y^{\omega}\).
 \end{enumerate}
\end{proposition}
\begin{proof}
 The proof is standard.
\end{proof}

Finally, since we proved in Theorem \ref{thmx:ZwrG-SelfSimilar} that 
\(\ZwrG\) is a self-similar group, from Theorem 
\ref{thm:Nekra-SimpleCommutator} it may be concluded that 
\([V_{m+1}(\ZwrG),V_{m+1}(\ZwrG)]\) is simple. We will prove that this 
faithful self-similar representation of \(\ZwrG\) embeds into a 
finitely presented simple group by combining the simplicity of 
\([V_{m+1}(\ZwrG),V_{m+1}(\ZwrG)]\) with the Kaloujnin--Krasner Theorem
and the following result due to Zaremsky.
\begin{proposition}\cite[Proposition 2.5]{Zaremsky2025b}\label{prop:Mergulho-Zaremsky}
 Let \(G\le\Am\) be a self-similar group and let \(V_m(G)\) be the associated SRN-group.
 For any finite group \(H\), the wreath product \([V_m(G),V_m(G)]\wr H\) embeds into \([V_m(G),V_m(G)]\).
\end{proposition}
\noindent
\subsection{Proof of Theorem \ref{thmx:SRN-FinitePresentation} - Item 1}
\subsubsection{Finite presentation}\label{subsubsec:finitepresentation}
Let \(G \leq \Am\) be a contracting self-similar group. The following result is an adaptation of Lemma~5.11 from
\cite{Nek2017} to the permutational wreath product \(\ZwrG\) and it establishes the finite
generation of the associated SRN-group \(V_{m+1}(\ZwrG)\). We will fix a finite presentation
\(\langle \mathcal{S} \mid \mathcal{R} \rangle\) for \(V_{m+1}\) (see \cite{Higman1974}) and define
\[\mathcal{S}_1\coloneqq \{L(g) \mid g\in\N\}\subseteq V_{m+1}(\ZwrG).\]
\begin{lemma}{\cite[Lemma 5.11]{Nek2017}}\label{lemma:SRN-Generation}
 If \(G\leq \Am\) is a contracting self-similar group, then the group \(V_{m+1}(\ZwrG)\) is generated by the set \(\mathcal{S}\cup \mathcal{S}_{1}\).
\end{lemma}
\begin{proof}
 By Lemma \ref{Lema3.1-Nekra}, for each non-empty word \(v\in \Y^{*}\), there exists an element \(h_{v}\in V_{m+1}\) such that \(h_{v}(vw)=x_{1}w\), for all \(w\in \Y^{\omega}\). Then, we have \(L_{v}(g)=h^{-1}_{v}L(g)h_{v}\), for every \(g\in V_{m+1}(\ZwrG)\). Then,
 \(L_{v}(g)\in\langle \mathcal{S}\cup \mathcal{S}_{1}\rangle\), for all \(g\in\N\) and non-empty \(v\in \Y^*\).
 Since each element from \(V_{m+1}(\ZwrG)\) can be represented by a table
 \[\begin{pmatrix}
 v_{1}\quad v_{2}\quad\ldots\quad v_{n} \\ \alpha_{1}\quad \alpha_{2}\quad\ldots\quad\alpha_{n}\\ u_{1}\quad u_{2}\quad\ldots\quad u_{n}
 \end{pmatrix},\]
 where \(v_{i},u_{i}\in \Y^{*}\) and \(\alpha_{i}\in\N\), it follows that
 \[g=L_{u_{1}}(\alpha_{1})L_{u_{2}}(\alpha_{2})\cdots L_{u_{n}}(\alpha_{n})\begin{pmatrix}
 v_{1}\quad v_{2} \quad \ldots \quad v_{n} \\ e \quad e \quad \ldots \quad e \\ u_{1}\quad u_{2}\quad\ldots\quad u_{n}
 \end{pmatrix}\in \langle \mathcal{S}\cup \mathcal{S}_{1}\rangle,\]
 with \(v_i, u_i \in \Y^*\) and \(\alpha_i \in \mathcal{N}\). It follows that every element lies in \(\langle\mathcal{S}\cup\mathcal{S}_1 \rangle\).

\end{proof}
\begin{rem}
 The previous proof is essentially the one of Nekrashevych in \cite{Nek2017}, the
 difference here is that we are working with the extended nucleus \(\N\) of \(\ZwrG\).
 We decided include this proof here in full to highlight the fact that the argument
 still works even if some state \(\alpha_i\) is a positive power of \(\gamma\).
 Indeed, if \(\alpha_i = \gamma^k\) for some \(k \in \mathbb{Z}_{>0}\), then
 \[L_{u_i}(\alpha_i) = L_{u_i}(\gamma^k) = L_{u_i}(\gamma)^k,\]
 which belongs in \(\langle\mathcal{S}\cup\mathcal{S}_1\rangle\).
\end{rem}
\begin{rem}
 In \cite{SkipperWitzelZaremsky2019}, the authors define an embedding map
\(\iota_{v}\colon \mathcal{G}\to V_m(\mathcal{G})\) given by
 \[(w)\iota_{v}(g)=\begin{cases}
 v(u)^{g},~\text{if}~w=vu~\text{for some}~u\in Y^{\omega} \\ w,~\text{otherwise},
\end{cases}\]
 with \(v\in Y^*\). This map has similar properties to the ones in Proposition \ref{prop:Propriedades-L} \cite{SkipperWitzelZaremsky2019,Zaremsky2026}.
 In this case, the authors assume that \(\mathcal{G}\leq \Am\) is a 
 finitely generated (coarsely diagonal) self-similar group and proved 
 that \(V_{m}(\mathcal{G})\) is finitely generated by the set 
 \(\mathcal{S}\cup\iota_{1}(\mathcal{G})\) \cite[Lemma 3.4]{SkipperWitzelZaremsky2019}.
\end{rem}
Now, to obtain a finite presentation, let us choose elements \(A_{x,y}, B_x \in V_{m+1}\),
for each pair \(x,y\in \Y\), satisfying
\[(yw)A_{x,y}=xyw \quad\text{and}\quad (x_1w)B_x=xw\]
for all \(w \in \Y^\omega\). We may assume \(B_{1}=e\). Let \(\overline{A}_{x,y}\) and \(\overline{B}_x\) be
words over \(S\) representing these elements.

Given \(y_1,\ldots,y_n,y \in \Y\), we define the word
\[\overline{A}_{y_1\cdots y_n, y} \coloneq \overline{A}_{y_n, y}\overline{A}_{y_{n-1}, y_n} \cdots \overline{A}_{y_1, y_2}.\]
Let \(A_{v,y}\) denote the image of \(\overline{A}_{v,y}\) in \(V_{m+1}\).

For every word \(v = y_1 \cdots y_n \in \Y^n\) with \(n \ge 2\) and every
\(g\in\mathcal{N}\), the element \(L_v(g) \in V_{m+1}(\ZwrG)\) satisfies the identity
\[L_v(g) = A_{y_1\cdots y_{n-1}, y_n}^{-1}\, B_{y_n}^{-1}\, L(g) \, B_{y_n} \, A_{y_1\cdots y_{n-1}, y_n}.\]
For each \(v = y_1\cdots y_n \in \Y^*\) with \(n\ge 2\) and each \(g \in \mathcal{N}\), consider the word \(\overline{L}_v(g)\) over \(\mathcal{S} \cup \mathcal{S}_1\) by
\[\overline{L}_v(g)\coloneqq\overline{A}_{y_1\cdots y_{n-1}, y_n}^{-1} \,
\overline{B}_{y_n}^{-1} \, L(g) \,
\overline{B}_{y_n} \,
\overline{A}_{y_1\cdots y_{n-1}, y_n},\]
where \(L(g)\) is a generator in \(S_1\).

Let \(\operatorname{Symm}((m+1)^n)\) denote the subgroup of \(V_{m+1}\) all of whose elements are of the form
\[
\begin{pmatrix}
 v_{1} & v_{2} & \ldots & v_{(m+1)^{n}} \\
 e & e & \ldots & e \\
 u_{1} & u_{2} & \ldots & u_{(m+1)^{n}}
 \end{pmatrix},
\]
where \(\{u_1,u_{2},\ldots,u_{(m+1)^{n}}\}=\{v_1,v_{2},\ldots,v_{(m+1)^{n}}\}=\Y^{n}\). This subgroup is naturally isomorphic to the symmetric group of degree \((m+1)^n\), with \(m = |Y|\).

For each \(x\in \Y\), choose a generating set \(W_x\), consisting of words over \(S\), for the pointwise stabilizer \((V_{m+1})_{x\Y^\omega}\) of the cone
\(C_x\). This stabilizer is isomorphic to the Higman--Thompson group
\(G_{m+1,m}\), and hence is finitely generated (see \cite{Higman1974}).
Therefore, \(W_x\) is a finite set. Next, we will consider the set
\(\mathcal{R}_1\) of the following finite collections of relations:
\begin{description}
 \item[(C) Commutation] Relations of the form
 \[[L_x(g_1),L_y(g_2)]=[L_{v_1}(g_1),L_{v_2}(g_2)]=[L(g),h]=e\]
 for all \(g_1,g_2,\in \N(G), x,y\in \Y, v_1,v_2\in \Y^2, h\in W_{x_1}\), where \(x\neq y\) and \(v_1\neq v_2\);
 \smallskip
 \item[(N) Nucleus] Relations of the form
 \[L(g_1)L(g_2)L(g_3) = e\]
 for all \(g_1,g_2,g_3\in \N(G)\) such that \(g_1g_2g_3=e\) in \(G\);
 \smallskip
 \item[(S) Splitting] Relations of the form
 \[L(g)=L_{11}(g|_{1}) L_{12}(g|_{2})\ldots L_{1(m+1)}(g|_{(m+1)})\cdot \bar{\sigma},\]
 for all \(g\in\N\), where \(\bar{\sigma}\) is a word in the generators \(S\) representing an element
 \(\sigma\in\Symm((m+1)^2)\) such that
 \[L(g)=L_{11}(g|_{_1})L_{12}(g|_{_2})\ldots L_{1(m+1)}(g|_{(m+1)})\sigma.\]
 In particular case of \(g=\gamma\), considering its wreath recursion \[\gamma=(\underbrace{\gamma,e,\dots, e}_{m_1}, \dots, \underbrace{\gamma,e,\dots, e}_{m_s}, \chi),\] these relations are
 \[L(\gamma)=L_{11}(\gamma) L_{12}(e) \ldots L_{1m_1}(e) \ldots L_{1(m_1 + \ldots + m_{s-1} + 1)}(\gamma) \ldots L_{1{m}}(e) L_{1(m+1)}(\chi),\]
 where \(\overline{\sigma}\) is omitted, since it represents \(id\in\Symm((m+1)^{2})\) in this case.
 \smallskip
 \item[(CB) Commutation Base] Relations of the form
 \[\bigl[L(\gamma)^{L(f)}, L(\gamma)^{L(g)}\bigr]= e,\]
 where \(f,g\in \N(G)\).
\end{description}
\begin{rem}
 Compared with \cite[p.~150]{Nek2017}, while
 relations \textbf{(N)} remains the same, relations \textbf{(C)} and \textbf{(S)}
 are adapted to the extended nucleus \(\N\) of \(\ZwrG\). In particular, relations
 \textbf{(C)} involve the states of the generator \(\gamma\) of the \(\mathbb{Z}\)-
 factor through the elements of \(\N\).
\end{rem}

Next we aim to show that the set of defining relations for the group \(V_{m+1}(\ZwrG)\) is \(\mathcal{R}\cup \mathcal{R}_1\)
and then conclude that this group is finitely presented.
Let \(\Gamma\) be the group defined by the presentation
\(\langle \mathcal{S}\cup \mathcal{S}_1\mid \mathcal{R}\cup \mathcal{R}_1\rangle\). Since all relations \(\mathcal{R}\cup \mathcal{R}_1\) hold in
\(V_{m+1}(\ZwrG)\), it suffices to prove that all relations in \(V_{m+1}(\ZwrG)\) also hold in \(\Gamma\).

We identify \(\langle \mathcal{S}\rangle \le \Gamma\) with its image in \(V_{m+1}\), since a group word in \(\mathcal{S}\) is trivial
in \(\Gamma\) if and only if it is trivial in \(V_{m+1}\). The next two lemmas are taken verbatim from \cite{Nek2017}.

\begin{lemma}\cite[Lemma 5.12]{Nek2017}\label{lem:Lema5.12-Nekra}
 Suppose that \(u,v\in \Y^{*}\) are non-empty. Let \(f\in V_{m+1}\) be such that \((uw)f=vw\), for all
 \(w\in \Y^{\omega}\). Then \(f^{-1}L_{u}(g)f=L_{v}(g)\) holds in \(\Gamma\).
\end{lemma}
\begin{proof}
 See \cite[p. 150]{Nek2017}.
\end{proof}
\begin{lemma}\cite[Lemma 5.13]{Nek2017}\label{lem:Lema5.13-Nekra}
 If \(u,v\in Y^{*}\) are incomparable, then \(\overline{L}_{u}(g)\) and \(\overline{L}_{u}(h)\)
 commute in \(\Gamma\), for all \(g,h\in\N\).
\end{lemma}
\begin{proof}
 See \cite[p. 151]{Nek2017} and consider the extended nucleus \(\N\) on the proof.
\end{proof}

Combining Lemma~\ref{lem:Lema5.12-Nekra} with the relations \textbf{(S)}, we obtain the following splitting relation:
\begin{description}
 \item[(S')] For each \(g\in\N\) and non-empty \(v\in \Y^*\), if \(g\in V_{m+1}\) satisfies
 \[L_v(g) = L_{v1}(g|_{1})L_{v2}(g|_{2}) \cdots L_{v(m+1)}(g|_{m+1})\sigma,\]
 we have
 \[L_v(g)=\overline{L}_{v1}(g|_{y_1})\overline{L}_{v2}(g|_{2})\cdots\overline{L}_{v(m+1)}(g|_{{m+1}})\sigma.\]
 For the particular case such that \(g=\gamma\),
 \[L_v(\gamma)=L_{v1}(\gamma|_{1})L_{vy_2}(e|_{y_2})\ldots L_{vm}(e|_{m}) L_{v(m+1)}(\chi|_{y_{m+1}})\]
 implies
 \[L_v(\gamma)=\overline{L}_{v1}(\gamma|_{1})\overline{L}_{v2}(e|_{2})\ldots \overline{L}_{vm}(e|_{m})\overline{L}_{v(m+1)}(\chi|_{{m+1}})\]
\end{description}

We need to prove the following new relation.
\begin{lemma}\label{lem:comuta-gamma}
 If \(g_1, \dots, g_r, h_1, \dots, h_s \in G \), then $$L(\gamma)^{n_1L(g_1)+ \dots + n_rL(g_r)} \text{ and }L(\gamma)^{t_1L(h_1)+ \dots + t_sL(h_s)}$$ commute in \(\Gamma\).
\end{lemma}
\begin{proof}
It is enough to prove 
\[\left[L(\gamma)^{L(g)}, L(\gamma)^{L(h)}\right] = e,\]
for $g, h \in G$. But
\[L(g)=L_{11}(g|_{1})\cdot L_{12}(g|_2)\ldots L_{1m}(g|_{m}) L_{1(m+1)}(g|_{m+1})\cdot L_1(\sigma_1), \, (m+1)\sigma_1 = m+1,\]

\[L(h)=L_{11}(h|_{1})\cdot L_{12}(h|_{2})\ldots L_{1y_m}(h|_{m})L_{1(m+1)}(h|_{(m+1)})\cdot L_1(\sigma_2), \, (m+1)\sigma_2 = m+1,\]

\[L(\gamma)=L_{11}(\gamma) L_{12}(e) \ldots L_{1m_1}(e) \ldots L_{1(m_1 + \ldots + m_{s-1} + 1)}(\gamma) \ldots L_{1{m}}(e) L_{1(m+1)}(\chi).\]
then
\[\left[L(\gamma)^{L(g)}, L(\gamma)^{L(h)}\right] = \left[L(\gamma^g), L(\gamma^h)\right] = L\left([\gamma^g, \gamma^h]\right) = \]
\[ = L((k_1, k_2, \dots, k_m, e)) = L(L_1(k_1) L_2(k_2) \dots L_m(k_m)L_{m+1}(e)),\]
with $[\gamma^g, \gamma^h] = (k_1, k_2, \dots, k_m, e)$, where $k_i = [\gamma^{g|_i}, \gamma^{h|_i}]$ if $i = (j)\sigma_1 = (j)\sigma_2$ and $j = 1, m_1 + 1, \dots, m_1 + \dots + m_{s-1} + 1$ and $k_i = e$ otherwise. Since $G$ is contracting self-similar, there exists a level that the states of $g$ and $h$ belong to $\mathcal{N}(G)$ and so 
\[\left[L(\gamma)^{L(g)}, L(\gamma)^{L(h)}\right] = e.\] 
\end{proof}
Finally, we will prove that any word over $\mathcal{S}\cup\mathcal{S}_1$ which is trivial in $V_{m+1}(\mathbb{Z}\wr_XG)$ is also trivial in \(\Gamma\). It follows from Lemmas~\ref{lemma:SRN-Generation}~and~\ref{lem:Lema5.12-Nekra}
that every element of \(\Gamma\) can be written as
\[L(g_1)^{h_1} L(g_2)^{h_2}\cdots L(g_n)^{h_n}h,\]
with \(h,h_i\in V_{m+1}\) and \(g_i\in\N\). Let \(n_1\) be such that the element \(h_n\) can be written as
\[h_{n}=\begin{pmatrix}
 v_{1} & v_{2} & \ldots & v_{m^{n_{1}}} \\
 e & e & \ldots & e \ \\
 u_{1} &u_{2}& \ldots & u_{m^{n_{1}}}
 \end{pmatrix},\]
where \(\{u_1, u_2,\ldots,u_{m^{n_1}}\}=\Y^{n_1}\). By relations \textbf{(S')} and Lemma~\ref{lem:Lema5.12-Nekra}, we write
\[L(g_n)=\prod_{v \in Y^{n_1-1}} \overline{L}_{1 v}(g_n|_v)~\alpha\]
for some \(\alpha \in \Symm((m+1)^{n_1})\). By Lemma~\ref{lem:Lema5.13-Nekra} the factors
\(\overline{L}_{x_1 v}(g_n |_v)\) commute with each other. Hence, for every \(v\in \Y^{n_1-1}\) there exists
\(i\in\{1,2,\ldots,(m+1)^{n_{1}}\}\) such that \(1v=u_i\). By Lemma~\ref{lem:Lema5.12-Nekra}, we have
\[\overline{L}_{1 v}(g_n |_v)^{h_n}=\overline{L}_{v_i}(g_n |_v),\]
so that \(L(g_n)^{h_n}\) can be rewritten as a product of \( \alpha^{h_n} \) followed by a product of elements
of the form \(\overline{L}_v(g_{v,n})\) for some \(v\in \Y^{n_1}\) and \(g|_{v,n}\in\N\).

Proceeding by induction, we can conclude that every element of \(\Gamma\) can be written as
\begin{equation}\label{eq:elementoemGamma}
 g=\overline{L}_{v_1}(g_1)~\overline{L}_{v_2}(g_2)\cdots\overline{L}_{v_\ell}(g_\ell)h
\end{equation}
\noindent for \(v_i\in Y^*\), \(g_i\in\N\) and \(h\in V_{m+1}\). Notice that in case of not
all words \(v_i\) have the same length, we may consider \(l_{\min}\coloneqq\min_i |v_i|\)
and choose an index \(i\) with \(|v_i|=l_{\min}\). Let
\[k\coloneqq\min\{|v_j|\mid|v_j|>l_{\min}\}
\]
be the second smallest length occurring among all \(v_j\). By \textbf{(S')}, we expand the factor
\(\overline{L}_{v_i}(g_i)\) as
\[\overline{L}_{v_i}(g_i)=\prod_{u\in \Y^{k-|v_i|}}\overline{L}_{v_i u}(g_i |_u)~\alpha_i,\]
with \(\alpha_i \in \Symm((m+1)^k)\). By Lemma~\ref{lem:Lema5.12-Nekra}, we can move \(\alpha_i\) to the front of
the product in (\ref{eq:elementoemGamma}). Observe that every new index \(v_i u\) has
length exactly \(k\). Hence the shortest length among the indices in the product strictly
increases from \(l_{\min}\) to at least \(k\), while the longest length remains unchanged. Iterating this procedure yields a strictly increasing sequence of minimal lengths, bounded above by the maximum length, so the process must terminate. The resulting expression has the same form as (\ref{eq:elementoemGamma}), but now all indices \(v_i\) have the same length.

Therefore, we may assume that in (\ref{eq:elementoemGamma}) all words \(v_i\) have the same length \(k\).
Under this assumption, the product
\[L_{v_1}(g_1) L_{v_2}(g_2) \cdots L_{v_\ell}(g_\ell)\]
does not alter the prefix of length \(k\) of any word \(w \in \Y^\omega\).
Since \(g\) is trivial in \(V_{m+1}(\ZwrG)\), it follows that \(h\) must also preserve those prefixes.
Hence \(h\) can be written as
\[h = \prod_{v \in \Y^k} L_v(h_v)\]
for some elements \(h_v\in V_{m}\). By Lemmas \ref{lem:Lema5.12-Nekra} and \ref{lem:Lema5.13-Nekra}, we may rearrange the factors in (\ref{eq:elementoemGamma}) so that
\[g=\prod_{v \in \Y^k} f_v,\]
where for each \(v \in \Y^k\),
\[
f_v = L_v(h_v) \, L_v(g_{v,1}) \, L_v(g_{v,2}) \cdots L_v(g_{v,\ell_v}),
\]
with \(h_v \in V_{m+1}\) and \(g_{v,i} \in \N\).

Notice that each \(f_v\) is trivial in \(V_{m+1}(\ZwrG)\), because the product over all \(v\) equals \(g = e\) and
the factors \(f_v\) act on disjoint cones (so they can only cancel within themselves).
The triviality of \(f_v\) implies that \(h_v\) lies in \(\Symm((m+1)^l)\) for some \(l\), and moreover the action of \(h_v g_{v,1} g_{v,2} \cdots g_{v,\ell_v}\) on the set \(\Y^l\) is the identity.
Consequently, using the splitting relations \textbf{(S')}, we can rewrite each \(f_v\) as a product of elements of the form \(L_{v u}(g)\) with \(u\in \Y^l\).
Thus, we may assume that all \(h_v\) are trivial.

Now, the fact that \(f_v\) is trivial reduces to the identity
\[g_{v,1} g_{v,2} \cdots g_{v,\ell_v} = e \quad \text{in } \ZwrG.\]
Since we may assume \(\ZwrG\) is generated by the extended nucleus \(\N\), the product may contain occurrences of
\(\gamma\). By Lemma~\ref{lem:comuta-gamma}, the elements \(\overline{L}(\gamma)^{\overline{L}(g)}\) commute
with each other in \(\Gamma\) for all \(g\in\N\). Hence, we can move all factors involving \(\gamma\) to the
front of the expression for \(f_v\), yielding
\[f_v = \overline{L}(\gamma)^t \cdot \prod_{i} \overline{L}_{v u_i}(g_i),\]
where \(g_i\in\N(G)\), \(u_i \in \Y^l\), and \(t\in \mathbb{Z}\).

Since \(f_v\) is trivial in \(V_{m+1}(\ZwrG)\), its image under the natural projection to \(V_{m+1}(G)\) must be
trivial. This projection sends \(\overline{L}(\gamma)\) to the identity (since \(\gamma\) projects to \(1\) in \(G\)), so we must have \(t=0\). Consequently,
\[\prod_{i}\overline{L}_{v u_i}(g_i)=e\]
and hence \(\prod g_i=e\) in \(G\).
All these reductions show that every relation of \(V_{m+1}(\ZwrG)\) follows from
\(\mathcal{R}\cup\mathcal{R}_1\), which is finite. Hence \(\Gamma\cong V_{m+1}(\ZwrG)\), and consequently \(V_{m+1}(\ZwrG)\) is finitely presented.

\begin{rem}\label{rem:Adaptacao-para-AwrG}
 The final step of the previous proof, where the projection of
 \(\overline{L}(\gamma)\) to \(V_m(G)\) is used to conclude that the exponent
 \(t\) must be zero, is easily adapted to the case where the base group is \(\mathbb{Z}^{d}\). Indeed, in this case the extended nucleus becomes
 \[
 \N\coloneqq \mathcal{N}(G) \cup Q(\gamma_1)\cup \dots\cup Q(\gamma_d),
 \]
 where \(\gamma_1, \dots, \gamma_d\) generate the \(\mathbb{Z}\)-factors. The extended nucleus remains finite, for the same argument. In the normalization of \(f_v\), any occurrence of either a generator \(\gamma_i\) is moved to the front, yielding a factor \(\overline{L}(\gamma_1)^{t_1} \cdots \overline{L}(\gamma_d)^{t_d}\), which is
 also annihilated by the projection. Since each \(\gamma_i\) projects to the identity
 in \(G\), the same projection argument makes all exponents \(t_i\) vanish.
 Consequently, the proof of finite presentability of \(V_{m+1}(\mathbb{Z} \wr_X G)\) extends \textit{verbatim} to \(V_{d(m+1)}(\mathbb{Z}^{d} \wr_X G)\).
\end{rem}

\subsubsection{Finite abelianization}\label{subsubsec:finiteabelizanization}
For a given \(G\leq\Am\) contracting self-similar group, the associated SRN-group of
\(V_{m+1}(\ZwrG)\) is finitely presented by Theorem \ref{thmx:SRN-FinitePresentation} and \([V_{m+1}(\ZwrG),V_{m+1}(\ZwrG)]\) is simple by Theorem \ref{thm:Nekra-SimpleCommutator}.
Nekrashevych's abelianization result \cite[Theorem 9.14]{Nek2004} yields that
\begin{enumerate}
 \item If \(m+1\) is even, then \(V_{m+1}(\ZwrG)^{ab}\) is obtained from the \((\ZwrG)^{ab}\) pass to quotient by the relations
 \begin{equation}
 \overline{g}=\overline{g|_1}+\overline{g|_2}+\ldots+\overline{g|_{m+1}},
 \end{equation}
where \(g\in\ZwrG\).
 \item If \(m+1\) is odd, \(V_{m+1}(\ZwrG)^{ab}\) is obtained from the \((\ZwrG)^{ab}\oplus \mathbb{Z}/2\mathbb{Z}\) pass to quotient by the relations
 \begin{equation}
 \overline{g} =(\overline{g|_1}+\overline{g|_2}+\ldots+\overline{g|_{m+1}})+\sgn(g),
 \end{equation}
 where \(g\in\ZwrG\) and, if \(g\) is even permutation on the
 first level of \(\T_m\), \(\sgn(g)\in\mathbb{Z}/2\mathbb{Z}\) is \(0\); otherwise, \(\sgn(g)=1\).
\end{enumerate}
In both cases, if \(V_{m+1}(\ZwrG)^{ab}\) is finite, then \([V_{m+1}(\ZwrG),V_{m+1}(\ZwrG)]\)
has finite index in \(V_{m+1}(\ZwrG)\) and, consequently, it is a finitely presented simple group.
If \(V_{m+1}(\ZwrG)^{ab}\) is infinite and it may not be finitely presented.
Thus, we adapt Zaremsky's strategy in \cite{Zaremsky2025b} for finitely presented self-similar groups.
The idea is to prove that we can produce an embedding of \(\ZwrG\) into \(\Amm\)
for a \(m'\in\mathbb{Z}_{>0}\), where \(m'=(m+1)r\), for some \(r\in\mathbb{Z}_{>0}\),
such that this copy of \(\ZwrG\) remains self-similar, for every \(G\) self-similar group,
but now the associated SRN-group \(V_{m'}(\ZwrG)\) has a finite abelianization. Consequently, the
commutator subgroup \([V_{m'}(\mathbb{Z}\wr_X G), V_{m'}(\mathbb{Z}\wr_X G)]\) has finite
index in \(V_{m'}(\mathbb{Z}\wr_X G)\), hence is a finitely presented simple group and we
are done.

Let us fix a self-similar representation of \(\ZwrG\) on the \((m+1)\)-ary tree, where \(m=|Y|\).
For each \(g\in\ZwrG\), let us consider its wreath recursion
\[(g|_1,\dots,g|_{m+1})\sigma.\]
Let \(r\) be an even positive integer. We construct a new faithful self-similar representation
of \(\ZwrG\) on the tree of degree \(m'\coloneqq r(m+1)\), by repeating each state \(r\) times. Explicitly, for each \(g\in\ZwrG\), we consider the following wreath recursion
\[(g|_1,\dots,g|_{m+1},\dots,g|_1,\dots,g|_{m+1})\sigma^{\oplus r},\]
where we have \(r(m+1)\) states and \(\sigma^{\oplus r}\in\Symm(r(m+1))\) acts like \(\sigma\) on \(\{1,\dots,m+1\}\) and acts like
\(\sigma\) in a natural way on \(\{m+2,\dots,2(m+1)\}\) and so on. This new representation is
faithful and self-similar since the original action of \(\ZwrG\) is. Abusing notation, we
continue to denote the group by \(\ZwrG\). Hence, we embedded it into \(\Amm\)
as a self-similar group. Like Zaremsky, we will work just with even \(m'\), which implies that
we take even \(r\in\mathbb{Z}_{>0}\).

Given that \(m'\) is even, it follows from Nekrashevych's abelianization result \cite[Theorem 9.14]{Nek2004} that \(V_{m'}(\ZwrG)^{ab}\) is the quotient of \((\ZwrG)^{ab}\) by the relations
\begin{equation}\label{eq:relacoes-abelianizacao}
 \overline{g}=r(\overline{g|_1}+\cdots+\overline{g|_{m+1}}),~\forall~g\in\ZwrG.
\end{equation}
Since replacing \(\ZwrG\) by \(\langle\N\rangle\) does not alter \(V_{m'}(\ZwrG)\), we may
assume the extended nucleus \(\N\) as a generating set of \(\ZwrG\). It will thus be
sufficient to impose (\ref{eq:relacoes-abelianizacao}) on the elements of it. For convenience,
we may represent the extended nucleus as \(\N=\{a_1,\dots,a_n\}\).

Now, we will prove that if we take the free abelian group generated by \(\overline{a}_1,\dots,\overline{a}_{n}\) and take the quotient with respect to the relation
\[\overline{a}_{i}=r((\overline{a}_{i})_{1}+\dots+(\overline{a}_{i})_{m+1}),~\forall~i\in\{1,\dots,n\}.\]
we get a finite abelian group for some even \(r\). This is equivalent to prove that if we take
\(\mathbb{R}^{n}\) with the basis \(\{\overline{a}_{1},\dots,\overline{a}_{n}\}\) and
take the quotient modulo the span of all the \(\overline{a}_{i}-r((\overline{a}_{i})_{1}+\dots+(\overline{a}_{i})_{m+1})\), we get the null vector
space for some even \(r\). Thus, consider a matrix \(M\in M_{n}(\mathbb{R})\) such that
the \(i\)-th column is \((\overline{a}_{i})_{1}+\dots+(\overline{a}_{i})_{m+1}\).

Relations (\ref{eq:relacoes-abelianizacao}) imply that
\[V_{m'}(\ZwrG)^{ab}\cong \mathbb{Z}^n / \operatorname{Im}(I_n-rM).\]
Hence, this abelianization is finite if and only if \(I_n-rM\) is invertible over
\(\mathbb{R}\), i.e., \(\det(I_n-rM)\neq 0\). Since \(M\) has only finitely many
eigenvalues, there are infinitely many even integers \(r\) such that it is not an
eigenvalue of \(M\). For such \(r\), we have \(V_{m'}(\ZwrG)^{ab}\) is finite.
\hfill\(\square\)

\begin{rem}
The same argument also holds for \(\mathbb{Z}^{d} \wr_X G\). If we 
consider the extended nucleus in this case as in
Remark \ref{rem:Adaptacao-para-AwrG}, the matrix \(M\) is defined 
using the states of all generators in \(\N\) and a similar
choice of an even integer \(r\) with
\(\det(I -rM)\neq0\) forces the abelianization of \(V_{m'}(\mathbb{Z}^{d} \wr_X G)\) to be finite.
\end{rem}

\subsection{Proof of Theorem \ref{thmx:SRN-FinitePresentation} - Item 2}
It follows from Item (1) of Theorem \ref{thmx:SRN-FinitePresentation}, 
that the group \(V_{d(m+1)}(\mathbb{Z}^{d}\wr_{X}G)^{ab}\)
is finite. This implies that \([V_{d(m+1)}(\mathbb{Z}^{d}\wr_{X}G),V_{d(m+1)}(\mathbb{Z}^{d}\wr_{X}G)]\) has finite index
in \(V_{d(m+1)}(\mathbb{Z}^{d}\wr_{X}G)\). Consequently, \([V_{d(m+1)}(\mathbb{Z}^{d}\wr_{X}G),V_{d(m+1)}(\mathbb{Z}^{d}\wr_{X}G)]\) is
finitely presented, since \(V_{d(m+1)}(\mathbb{Z}^{d}\wr_{X}G)\) is finitely presented, and the group
\[H\coloneqq (\mathbb{Z}^{d}\wr_{X}G)\cap [V_{d(m+1)}(\mathbb{Z}^{d}\wr_{X}G),V_{d(m+1)}(\mathbb{Z}^{d}\wr_{X}G)]\] 
has
finite index in \(\mathbb{Z}^{d}\wr_{X}G\). Considering the normal core of \(H\) in \(\mathbb{Z}^{d}\wr_{X}G\),
we get a normal subgroup \(H_{\mathbb{Z}^{d}\wr_{X}G} \trianglelefteq \mathbb{Z}^{d}\wr_{X}G\) of finite index.
By the Kaloujnine--Krasner Theorem,
\[\mathbb{Z}^{d}\wr_{X}G \hookrightarrow H_{\mathbb{Z}^{d}\wr_{X}G}\wr ((\mathbb{Z}^{d}\wr_{X}G)/H_{\mathbb{Z}^{d}\wr_{X}G}).\]
Since \((\mathbb{Z}^{d}\wr_{X}G)/H_{\mathbb{Z}^{d}\wr_{X}G}\) is finite, by Proposition \ref{prop:Mergulho-Zaremsky}, we have
\[[V_{d(m+1)}(\mathbb{Z}^{d}\wr_{X}G),V_{d(m+1)}(\mathbb{Z}^{d}\wr_{X}G)]\wr\left(\dfrac{\mathbb{Z}^{d}\wr_{X}G}{H_{\mathbb{Z}^{d}\wr_{X}G}}\right)\hookrightarrow [V_{d(m+1)}(\mathbb{Z}^{d}\wr_{X}G),V_{d(m+1)}(\mathbb{Z}^{d}\wr_{X}G)].\]
Thus,
\[\mathbb{Z}^{d}\wr_{X}G\hookrightarrow [V_{d(m+1)}(\mathbb{Z}^{d}\wr_{X}G),V_{d(m+1)}(\mathbb{Z}^{d}\wr_{X}G)].\]
Therefore, \(\mathbb{Z}^{d}\wr_{X}G\) embeds into a finitely presented simple group.

\hfill\(\square\)

\subsection{Examples} $(i)$ By Example \ref{4.8} and the Belk and Matucci result, the group \(C_2 \wr(C_2 \wr\mathbb{Z})\) embed in a finitely presented simple group. By Example \ref{4.8} and Theorem \ref{thmx:SRN-FinitePresentation}, the group \(\mathbb{Z}\wr(C_2\wr\mathbb{Z})\) embed in a finitely presented simple group. 

\,

\noindent $(ii)$ Nathan Corwin proved that, for every \(d\in\mathbb{Z}_{>0}\), 
\(\mathbb{Z}^{}\wr\mathbb{Z}^{d}\)
embeds into Brin's group \(dV\) \cite[Theorem 27]{Corwin2013}, 
which is a finitely presented simple group. By Theorem \ref{thmx:ZwrG-SelfSimilar}, \(\mathbb{Z}\wr\mathbb{Z}^{d}\) has a self-similar
representation generated by the following elements of \(\mathcal{A}_{3}\)
\begin{equation*}
\gamma=(\gamma _{},e,\alpha _{1}),
\alpha _{1}=(e,\alpha _{d},e)(0\,1),\,\alpha _{2}=(\alpha
_{1},\alpha _{1},e),\dots,\alpha _{d}=(\alpha _{d-1},\alpha
_{d-1},e).
\end{equation*}%
Applying Nekrashevych's abelianization result \cite[Theorem 9.14]{Nek2004}, we get 
\[V_{3}(\mathbb{Z}\wr\mathbb{Z}^{d})=[V_{3}(\mathbb{Z}\wr\mathbb{Z}^{d}),V_{3}(\mathbb{Z}\wr\mathbb{Z}^{d})],\] \noindent which means that \(V_{3}(\mathbb{Z}\wr\mathbb{Z}^{d})\) is a finitely presented simple group. Thus, for every \(d\in\mathbb{Z}_{>0}\), \(\mathbb{Z}\wr\mathbb{Z}^{d}\) embeds into
the finitely presented simple group the SRN-group \(V_{3}(\mathbb{Z}\wr\mathbb{Z}^{d})\), where the degree of the tree is
fixed \(m=3\).

Considering this last example and that Corwin conjectured that 
\(\mathbb{Z}\wr\mathbb{Z}^{d+1}\) does not embed into 
\(dV\) \cite[Conjecture 3]{Corwin2013}, we have the following question.
\begin{question}
    Are the SRN-groups \(V_{3}(\mathbb{Z}\wr\mathbb{Z}^{d})\) and
    \(V_{3}(\mathbb{Z}\wr\mathbb{Z}^{l})\) isomorphic, for every \(d,l\in\mathbb{Z}_{>0}\)?
\end{question}

\subsection*{Declaration of AI Assistance}
During the preparation of this manuscript, the authors used an AI-assisted tool (ChatGPT-5) for language revision, formatting, and to assist in the development of some argumentative insights (Theorem \ref{thmx:PSL-FinitelyPresentedTransitiveSelfSimilar}, item (2)). The AI was not used to generate or validate mathematical content, theorems, or proofs, and is not listed as an author. The authors reviewed all AI-generated suggestions and assume full responsibility for the final content.

\end{document}